\documentclass{ws-jktr}

\usepackage{amssymb}
\usepackage{amsfonts}
\usepackage[all]{xy}
\usepackage{color}
\usepackage{url}
\usepackage{graphicx}
\usepackage{tikz}
\usepackage{subfigure}

\def\X{X}
\def\Y{Y}
\def\P{P}
\def\lmlt#1{\mathrm{LMlt}(#1)}

\def\sym#1{\mathbb S_{#1}}
\def\aut#1{\mathrm{Aut}(#1)}
\def\ker#1{\mathrm{ker}(#1)}

\def\gal#1{\mathcal{Q}(#1)}

\def\setof#1#2{\{#1\,\mid\,#2\}}
\def\genof#1#2{\langle#1\,\mid\,#2\rangle}
\def\ldiv{\backslash}
\def\lcm{\mathrm{lcm}}
\def\gcd{\mathrm{gcd}}

\begin{document}

\markboth{Marco Bonatto and Petr Vojt\v{e}chovsk\'y}
{Simply connected latin quandles}

\catchline{}{}{}{}{}

\title{SIMPLY CONNECTED LATIN QUANDLES}

\author{MARCO BONATTO}

\address{KA MFF UK, Charles University\\Sokolovsk\'{a} 83, 18675 Praha, Czech Republic}

\author{PETR VOJT\v{E}CHOVSK\'Y}
\address{Department of Mathematics, University of Denver\\2390 S York St, Denver, Colorado 80208, U.S.A.}

\maketitle

\begin{abstract}
A (left) quandle is connected if its left translations generate a group that acts transitively on the underlying set. In 2014, Eisermann introduced the concept of quandle coverings, corresponding to constant quandle cocycles of Andruskiewitsch and Gra\~{n}a. A connected quandle is \emph{simply connected} if it has no nontrivial coverings, or, equivalently, if all its second constant cohomology sets with coefficients in symmetric groups are trivial.

In this paper we develop a combinatorial approach to constant cohomology. We prove that connected quandles that are affine over cyclic groups are simply connected (extending a result of Gra\~{n}a for quandles of prime size) and that finite doubly transitive quandles of order different from $4$ are simply connected. We also consider constant cohomology with coefficients in arbitrary groups.
\end{abstract}

\keywords{Quandle, connected quandle, latin quandle, quandle cohomology, non-abelian cohomology, constant cocycle, simply connected quandle, covering.}

\ccode{Mathematics Subject Classification 2000: Primary 05C38, 15A15; Secondary 05A15, 15A18.}


\section{Introduction}

Quandles were introduced by Joyce \cite{J} in 1982 as algebraic objects whose elements can be used to color oriented knots. In more detail, let $K$ be an oriented knot and $X=(X,\triangleright,\overline\triangleright)$ a set with two binary operations. Given a diagram of $K$, an assignment of elements of $X$  to arcs of $K$ is \emph{consistent} if the relations
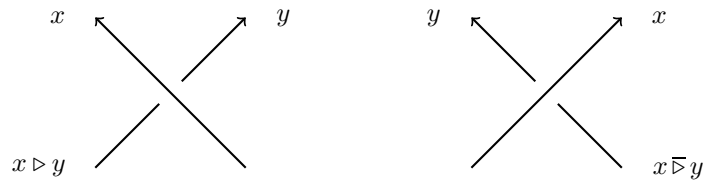
\begin{displaymath}
\begin{tikzpicture}[thick]
    \draw[->] (1,-1) -- (-1,1);
    \draw (-1,-1) -- (-0.15,-0.15);
    \draw[->] (0.15,0.15) -- (1,1);
    \node at (-1.5,1) {$x$};
    \node at (1.5,1) {$y$};
    \node at (-1.75,-1) {$x\triangleright y$};
    \draw (6,-1) -- (5.15,-0.15);
    \draw[->] (4.85,0.15) -- (4,1);
    \draw[->] (4,-1) -- (6,1);
    \node at (3.5,1) {$y$};
    \node at (6.5,1) {$x$};
    \node at (6.75,-1) {$x\,\overline\triangleright\, y$};
\end{tikzpicture}
\end{displaymath}
are satisfied at every crossing of the diagram. If, in addition, the assignment remains consistent when Reidemeister moves are applied to the diagram of $K$ in all possible ways, we say that the assignment is a \emph{coloring of $K$}. We denote by $\mathrm{col}_X(K)$ the number of colorings of $K$ by elements of $X$ in which more than one element is used.

Whether $\mathrm{col}_X(K)\ne 0$ is a delicate question that depends on both $K$ and $X=(X,\triangleright,\overline\triangleright)$. However, when the consistency conditions imposed by Reidemeister moves are universally quantified, they force $\overline\triangleright$ to be the inverse operation to $\triangleright$ (in the sense that $x\triangleright(x\,\overline\triangleright\, y)=y$ and $x\,\overline\triangleright\,(x\triangleright y)=y$) and they force precisely the quandle axioms onto $\triangleright$. In particular, if $(X,\triangleright)$ is a quandle and there is a consistent assignment of elements of $X$ to arcs of a diagram of $K$, it is also a coloring of $K$. It is therefore customary to color arcs of oriented knots by elements of quandles.

For an oriented knot $K$, the knot quandle of $K$ is the quandle freely generated by the oriented arcs of $K$ subject only to the above crossing relations. It was shown early on by Joyce \cite{J} and Matveev \cite{Matveev} that knot quandles are complete invariants of oriented knots up to mirror image and reversal of orientation.

It has been conjectured that for any finite collection $\mathcal K$ of oriented knots there exists a finite collection of finite quandles $X_1$, $\dots$, $X_n$ such that the $n$-tuples $(\mathrm{col}_{X_1}(K),\dots,\mathrm{col}_{X_n}(K))$ distinguish the knots of $\mathcal K$ up to mirror image and reversal of orientation. (See \cite{CESY} for a more precise formulation.) This conjecture has been verified in \cite{CESY} for all prime knots with at most $12$ crossings, using a certain collection of $26$ quandles. Many oriented knots can be distinguished by a coarser invariant, namely by merely checking whether $\mathrm{col}_X(K)\ne 0$ \cite{Davidknots, K}.

Whenever an oriented knot $K$ is colored by a quandle $X$, the elements of $X$ actually used in the coloring form a connected subquandle of $X$. Consequently, for the purposes of quandle colorings, it is sufficient to consider connected quandles. Although far from settled, the theory of connected quandles is better understood than the general case \cite{HSV}, and connected quandles have been classified up to size $47$ \cite{HSV, Vendramin_connected}.

\medskip

Our work is primarily concerned with simply connected quandles which can be defined in several equivalent ways.

In a long paper \cite{E}, Eisermann developed the theory of quandle coverings and offered a complete categorical characterization of coverings for a given quandle. In \cite{AG}, Andruskiewitsch and Gra\~{n}a introduced an extension theory for quandles. Their (dynamical) quandle cocycles can be used to color arcs of oriented knots, giving rise to knot invariants \cite{CEGS}. (These were first defined in \cite{CJKLS} as an analog of the Dijkgraaf-Witten invariants for 3-manifolds \cite{DW}.) The quandle cocycle condition (see Definition \ref{Df:QC}) is rather difficult to work with, involving a $3$-parameter mapping. Constant quandle cocycles are precisely those quandle cocycles in which one parameter is superfluous (see Definition \ref{Df:CQC}). Even constant quandle cocycles yield powerful knot invariants \cite[Section 5]{CEGS}.

The following conditions are equivalent for a connected quandle $X$ (see Proposition \ref{Pr:SC}):
\begin{itemize}
\item the fundamental group of $X$ is trivial,
\item every covering of $X$ is equivalent to the trivial covering of $X$ over some set $S$,
\item for every set $S$, the second constant cohomology set of $X$ with coefficients in the symmetric group $\sym{S}$ is trivial.
\end{itemize}
If a connected quandle $X$ satisfies any of these equivalent conditions, we say that $X$ is \emph{simply connected}.

\medskip

In this paper we develop a combinatorial approach to constant cohomology with emphasis on simply connected quandles. From an algebraic point of view, our main result is as follows:

\begin{theorem}\label{Th:Main}
Let $X$ be a finite connected quandle that is affine over a cyclic group, or a finite doubly transitive quandle of size different from $4$. Then $X$ is simply connected.
\end{theorem}

We offer two proofs of Theorem \ref{Th:Main}. The first proof is combinatorial in nature and mostly self-contained. The second proof (whose main idea was suggested to us by the referee of an earlier version of this paper) is much shorter and relies on an explicit description of the fundamental group of affine quandles from an unpublished note of Clauwens \cite{Clauwens}.

We also investigate constant cohomology with coefficients in arbitrary groups and we prove:

\begin{theorem}\label{Th:Main2}
Let $X$ be a latin quandle. Then the following conditions are equivalent:
\begin{itemize}
\item[(i)] $X$ is simply connected,
\item[(ii)] $H^2_c(X,G)=1$ for every group $G$.
\end{itemize}
\end{theorem}

We can then easily obtain the following knot-theoretical corollary:

\begin{corollary}\label{Cr:Main}
Let $X$ be a simply connected latin quandle. Then every conjugacy quandle cocycle invariant based on $X$ is trivial.
\end{corollary}

\medskip

The paper is organized as follows. Basic results about quandles and their extensions, coverings, cohomology and constant cohomology are recalled in Section \ref{Sc:QuandleExtensions}. Let $X$ be a latin quandle, $u\in X$, and let $S$ be a nonempty set. Every constant quandle cocycle $\beta:X\times X\to \sym{S}$ with coefficients in the symmetric group $\sym{S}$ is cohomologous to a normalized (constant) quandle cocycle $\beta_u$ satisfying $\beta_u(x,u)=1$ for every $x\in X$, as is recalled in Section \ref{Sc:NormalizedCocycles}. In Section \ref{Sc:ThreeBijections} we introduce three bijections $f$, $g$, $h$ of $X\times X$ under which every $u$-normalized cocycle $\beta_u$ is invariant, that is, $\beta_u(k(x,y)) =\beta_u(x,y)$ for $k\in\{f,g,h\}$. To prove that a given connected quandle is simply connected, it then suffices to show that for every $(x,y)\in X\times X$ there exists some $(x_0,y_0)\in X\times X$ and $k\in\langle f,g,h\rangle$ such that $\beta_u(x_0,y_0)=1$ and $k(x_0,y_0)=(x,y)$. We therefore study the orbits of $f$, $g$, $h$ in Section \ref{Sc:Orbits}, and again in Section \ref{Sc:OrbitsOnAffine} in the restricted case of connected affine quandles. Theorem \ref{Th:Main} is proved in Section \ref{Sc:Classes}. Clauwens' explicit description of the fundamental group of affine quandles is recalled in Section \ref{Sc:Clauwens} and then Theorem \ref{Th:Main} is proved once more. Finally, constant cohomology with coefficients in arbitrary groups is introduced in Section \ref{Sc:ArbitraryGroup}, where we also prove Theorem \ref{Th:Main2} and Corollary \ref{Cr:Main}.

\section{Basic results and quandle extensions}\label{Sc:QuandleExtensions}

\subsection{Quandles}

For a groupoid $(X,\cdot)$ and $x\in X$, let
\begin{align*}
    &L_x:X\to X,\quad x\mapsto x\cdot y,\\
    &R_x:X\to X,\quad x\mapsto y\cdot x
\end{align*}
be the \emph{left translation} by $x$ and the \emph{right translation} by $x$, respectively.

We will often suppress the binary operation while talking about groupoids and denote them by $X$ rather than by $(X,\cdot)$. We denote by $\aut{X}$ the automorphism group of $X$. When $X$ is merely a set, then $\aut{X}=\sym{X}$, the symmetric group on $X$.

\begin{definition}
A groupoid $(X,\triangleright)$ is a \emph{quandle} if it is
a left quasigroup that is left distributive and idempotent. That is, $(X,\triangleright)$ is a quandle if it satisfies the following axioms:
\begin{gather}
	L_x\in\sym{X},\label{LQG}\\
    x\triangleright(y\triangleright z)  =  (x\triangleright y)\triangleright (x\triangleright z),\label{LD}\\
    x\triangleright x  =  x\label{ID}
\end{gather}
for every $x$, $y$, $z\in X$.
\end{definition}

Note that the identity \eqref{LQG} is equivalent to $\X$ having a left division operation defined by
\begin{displaymath}
   x\ldiv y = L_x^{-1}(y),
\end{displaymath}
and that the two identities \eqref{LQG} and \eqref{LD} jointly state that $L_x\in\aut{X}$ for every $x\in X$.

Every quandle is \emph{flexible}, indeed, $x \triangleright(y\triangleright x) = (x\triangleright y)\triangleright (x\triangleright x) = (x\triangleright y)\triangleright x$, and it is therefore safe to write $x\triangleright y\triangleright x$.

For any groupoid $\X$ and $\varphi\in\aut{\X}$ we have $\varphi L_x\varphi^{-1} = L_{\varphi(x)}$ for every $x\in X$. In particular, if $\X$ is a quandle then
\begin{equation}\label{CONJ}
    L_y L_x L_y^{-1}=L_{y\triangleright x}\text{ for every $x$, $y\in X$}.
\end{equation}

\begin{example}
Quandles appear naturally as the following examples illustrate.
\begin{enumerate}
\item[(i)] The one element groupoid is called the \emph{trivial quandle}. More generally, any projection groupoid on a set $X$ (that is, a groupoid satisfying $x\triangleright y=y$ for every $x$, $y\in X$) is a quandle, the \emph{projection quandle} $\P_X$ over $X$.
\item[(ii)] Let $G$ be a group and $H$ a union of (some) conjugacy classes of $G$. For $x$, $y\in H$, let $x\triangleright y = xyx^{-1}$. Then $(H,\triangleright)$ is a quandle, the \emph{conjugation quandle} on $H$.
\item[(iii)] Let $G$ be a group, $\alpha\in\aut{G}$ and $H\le \mathrm{Fix}(\alpha)=\setof{x\in G}{\alpha(x)=x}$. Let $G/H$ be the set of left cosets $\setof{xH}{x\in G}$. Then $G/H$ with multiplication
    \begin{displaymath}
        xH\triangleright yH=x\alpha(x^{-1}y)H
    \end{displaymath}
    is a quandle $\gal{G,H,\alpha}=(G/H,\triangleright)$, called the \emph{coset quandle} (also known as \emph{homogeneous quandle} or \emph{Galkin quandle}).
\item[(iv)] A coset quandle $\gal{G,H,\alpha}$ with $H=1$ is called \emph{principal} and will be denoted by $\gal{G,\alpha}$. If, in addition, $G$ is an abelian group, then $\gal{G,\alpha}$ is an \emph{affine quandle}.
\end{enumerate}
\end{example}

\begin{definition}
For a quandle $X$, we call the set
\begin{displaymath}
    L_\X=\setof{L_x}{x\in X}
\end{displaymath}
the \emph{left section} of $\X$, and the group
\begin{displaymath}
    \lmlt{\X}=\langle L_X\rangle\le\aut{X}
\end{displaymath}
the \emph{left multiplication group} of $\X$. The group $\lmlt{X}$ is often denoted by $\mathrm{Inn}(X)$ and called the \emph{inner automorphism group} of $X$.
\end{definition}

The left section $L_X$ is closed under conjugation by \eqref{CONJ}, and the corresponding conjugation quandle on $L_X$ will be denoted by $L(X)$. Note that the mapping $X\to L(X)$, $x\mapsto L_x$ is a homomorphism of quandles.

\begin{definition}
A quandle $X$ is \emph{latin} if $R_x \in \sym{X}$ for every $x\in X$.
\end{definition}

In a latin quandle we can define right division by $x/y = R_y^{-1}(x)$. A latin quandle $X$ is therefore a quasigroup $(X,\triangleright,\ldiv,/)$ in which the multiplication $\triangleright$ is left distributive and idempotent. As in any quasigroup, a homomorphism of a latin quandle $(X,\triangleright)$ is automatically a homomorphism of $(X,\triangleright,\ldiv,/)$. For instance, $x\triangleright (y/z) = L_x(y/z) = L_x(y)/L_x(z) = (x/y)\triangleright (x/z)$ holds in any latin quandle.

\begin{definition}
A quandle $\X$ is \emph{connected} (or, sometimes, \emph{transitive}) if $\lmlt{\X}$ acts transitively on $X$, and \emph{doubly transitive} if $\lmlt{X}$ acts doubly transitively on $X$.
\end{definition}

All latin quandles are connected. Indeed, if $X$ is a latin quandle and $x$, $y\in X$, then $L_{x/y}(y)=x$. All finite quandles can be built from connected quandles but the extension theory is not well understood, cf. \cite[Proposition 1.17]{AG} or \cite{polish}.

In order to simplify notation, we will from now on denote the quandle multiplication by $\cdot$ or by juxtaposition rather than by $\triangleright$.

\subsection{Quandle extensions}

The notion of quandle extensions was introduced in \cite{AG}. Let $X$ be a groupoid, $S$ a nonempty set, and suppose that $(X\times S,\cdot)$ is a groupoid. Then the canonical projection $\pi:X\times S\to X$, $(x,s)\mapsto x$ is a homomorphism of groupoids if and only if there is a mapping $\beta:X\times X\times S\times S\to S$ such that
\begin{displaymath}
    (x,s)\cdot (y,t) = (xy,\beta(x,y,s,t)).
\end{displaymath}
We will then denote $(X\times S,\cdot)$ by $X\times_\beta S$.

Suppose now that $X$ is a quandle. It is then easy to see that $X\times_\beta S$ is also a quandle if and only if $\beta(x,y,s):S\to S$, $t\mapsto \beta(x,y,s,t)$ is a bijection for every $x$, $y\in X$, $s\in S$, and the quandle cocycle conditions
\begin{eqnarray}
    \beta ( xy,xz,\beta(x,y,s)(t)) \beta ( x,z,s) &= &\beta (x,yz,s) \beta ( y,z,t),\label{cocycle condition}\\
    \beta ( x,x,s)(s) & = & s\label{quandle condition}
\end{eqnarray}
hold for every $x$, $y$, $z\in X$ and every $s$, $t\in S$. In the context of quandles, we will therefore consider $\beta$ as a mapping $X\times X\times S\to\sym{S}$.

\begin{definition}[{{\cite[Definition 2.2]{AG}}}]\label{Df:QC} Let $\X$ be a quandle and $S$ a nonempty set. A mapping $\beta:X\times X\times S\to\sym{S}$ is a \emph{quandle cocycle} if \eqref{cocycle condition} and \eqref{quandle condition} hold. The set of all quandle cocycles $X\times X\times S\to\sym{S}$ will be denoted by $Z^2(\X,\sym{S})$.
\end{definition}

\begin{proposition}[\cite{AG}]\label{Pr:Extension}
The following conditions are equivalent for quandles $\X$, $\Y$:
\begin{enumerate}
\item[(i)] $Y$ is a quandle defined on $X\times S$ for some set $S$ and the canonical projection $Y\to X$ is a quandle homomorphism.
\item[(ii)] $Y\cong X\times_\beta S$ for some set $S$ and some quandle cocycle $\beta\in Z^2(\X,\sym{S})$.
\item[(iii)] $X\cong Y/\alpha$ for a uniform congruence $\alpha$ of $\Y$, that is, a congruence $\alpha$ of $\Y$ such that all blocks of $\alpha$ have the same cardinality.
\end{enumerate}
\end{proposition}

\begin{proof}
We have already shown above that (i) and (ii) are equivalent. Suppose that (ii) holds and $Y=X \times_\beta S$ for some $\beta\in Z^2(X,\sym{S})$. Then $X\cong Y/\ker{\pi}$, where $\pi: X\times_\beta S\to X$ is the canonical projection. Clearly, $\ker{\pi}$ is a uniform congruence, each block having cardinality $|S|$. Conversely, let $\alpha$ be a uniform congruence on $\Y$, $X=Y/\alpha$, and let $S$ be a set of the same cardinality as any of the blocks of $\alpha$. Let $\setof{h_{[x]}:[x]\to S}{[x]\in\X}$ be a family of bijections indexed by the blocks of $\alpha$. Then the mapping $\beta:X\times X\times S\to\sym{S}$ defined by
\begin{displaymath}
    \beta ([x],[y],s) = h_{[xy]} L_{h_{[x]}^{-1}(s)}h_{[y]}^{-1}
\end{displaymath}
is a quandle cocycle and the mapping
\begin{displaymath}
    Y\to X\times_\beta S,\quad x\mapsto ([x] ,h_{[x]}(x))
\end{displaymath}
is an isomorphism of quandles.
\end{proof}

We therefore define:

\begin{definition}
Let $\X$, $\Y$ be quandles. Then $\Y$ is an \emph{extension} of $\X$ if $\X\cong \Y/\alpha$ for some uniform congruence $\alpha$ of $\Y$.
\end{definition}

For a quandle $\X$, let $\mathcal H(\X)$ denote the class of all homomorphic images of $\X$. We have:

\begin{proposition}
Let $\Y$ be a connected quandle. Then the following conditions are equivalent for a quandle $\X$:
\begin{enumerate}
\item[(i)] $\X\in\mathcal H(\Y)$,
\item[(ii)] $\Y$ is an extension of $\X$.
\end{enumerate}
\end{proposition}
\begin{proof}
By the Fundamental Homomorphism Theorem, $\X\in\mathcal H(\Y)$ if and only if $\X\cong \Y/\alpha$ for some congruence $\alpha$ of $\Y$. Since $\Y$ is connected, it is easy to show that every congruence of $\Y$ is uniform.
\end{proof}

The following equivalence relation partitions $Z^2(\X,\sym{S})$ so that any two cocycles in the same block give rise to isomorphic quandles. In a suitably defined category of quandle extensions (see \cite{AG} or Proposition \ref{Pr:Categorical}), two cocycles belong to the same block of this partition if and only if the two quandle extensions are isomorphic.

\begin{definition}
Let $\X$ be a quandle and $S$ a nonempty set. We say that $\beta$, $\beta'\in Z^2(\X,\sym{S})$ are \emph{cohomologous}, and we write $\beta\sim\beta'$, if there exists a mapping
\begin{equation*}
    \gamma :X\to \sym{S}
\end{equation*}
such that
\begin{equation*}
    \beta'(x,y,s) =\gamma(xy) \beta(x,y,\gamma(x)^{-1}(s)) \gamma(y)^{-1}
\end{equation*}
holds for every $x$, $y\in X$ and $s\in S$. The factor set
\begin{equation*}
    H^2(\X,\sym{S}) =Z^2(\X,\sym{S})/\sim
\end{equation*}
is the \emph{second (non-abelian) cohomology set} of $\X$ with coefficients in $\sym{S}$.
\end{definition}

The following result makes clear the relationship between cohomologous cocycles and isomorphisms of quandle extensions.

\begin{proposition}[{{\cite[pp. 194--195]{AG}}}]\label{Pr:Categorical}
The following conditions are equivalent for a quandle $X$, a set $S$ and cocycles $\beta$, $\beta'\in Z^2(X,\sym{S})$:
\begin{enumerate}
\item[(i)] $\beta\sim \beta^{\prime}$,
\item[(ii)] there exists an isomorphism $\phi: \X\times_{\beta} S\longrightarrow\X\times_{\beta^{\prime}} S$ such that the following diagram is commutative
\begin{displaymath}
    \xymatrixcolsep{63pt}\xymatrixrowsep{30pt}\xymatrix{
    \X\times_\beta S\ar[r]^{\pi}
    \ar[d]^{\phi} & \X \\
    \X\times_{\beta^{\prime}} S \ar[ur]^{\pi}&}.
\end{displaymath}
\end{enumerate}
\end{proposition}

\subsection{Quandle coverings and constant cocycles}

We are interested in a special class of quandle extensions, so-called quandle coverings.

\begin{definition}[{{\cite[Definition 1.4]{E}}}]\label{Df:Covering}
A connected quandle $\Y$ is a \emph{covering} of a quandle $\X$ if there is a surjective quandle homomorphism $f:\Y\to \X$ such that the left translations $L_x$, $L_y$ of $\Y$ coincide whenever $f(x)=f(y)$.
\end{definition}

For a quandle $Y$, let $\ker{L_Y}$ denote the equivalence relation on $Y$ induced by equality in the left section $L_Y$, that is, $(x,y)\in\ker{L_Y}$ if and only if $L_x=L_y$. Then $\ker{L_Y}$ is in fact a congruence on $Y$, thanks to \eqref{CONJ}. Moreover, if $Y$ is a connected quandle then $\ker{L_Y}$ is a uniform congruence, and hence $Y$ is an extension of $Y/\ker{L_Y}$. Therefore, a connected quandle $\Y$ is a covering of a quandle $\X$ if and only if $\X\cong \Y/\alpha$, where $\alpha$ is some uniform congruence of $Y$ that refines $\ker{L_Y}$. Here, we say that a congruence $\alpha$ refines a congruence $\beta$ if $(x,y)\in\beta$ whenever $(x,y)\in\alpha$.

Here are some nontrivial examples of quandle coverings:

\begin{proposition}
Let $\X_1=\gal{G,H_1,\alpha}$ and $\X_2=\gal{G,H_2,\alpha}$ be two coset quandles such that $H_1\leq H_2$. Then $\X_1$ is a covering of $\X_2$.
\end{proposition}

\begin{proof}
Define $\psi :X_1\to X_2$ by $\psi (xH_1)=xH_2$. The mapping $\psi$ is surjective and every block of $\ker{\psi}$ has the same cardinality as $H_2/H_1$. For $x$, $y\in G$ we have
\begin{displaymath}
    \psi  ( xH_1\cdot yH_1) = \psi   ( x\alpha ( x^{-1}y) H_1) = x\alpha ( x^{-1}y) H_2= xH_2\cdot yH_2= \psi( xH_1) \cdot \psi  (yH_1),
\end{displaymath}
so $\psi $ is a homomorphism.

Suppose that $\psi (xH_1)=\psi (yH_1)$, i.e., $x=yh$ for some $h\in H_2\leq \mathrm{Fix}(\alpha)$. Then $x\alpha(x^{-1})=yh\alpha(h^{-1}y^{-1})=y\alpha(y^{-1})$ and thus
\begin{displaymath}
    xH_1\cdot zH_1 =x\alpha(x^{-1}z)H_1=y\alpha(y^{-1}z) H_1=yH_1\cdot zH_1
\end{displaymath}
for every $z\in G$. This shows that $L_{xH_1} = L_{yH_1}$ in $X_1$, so $X_1$ is a covering of $X_2$.
\end{proof}

We proceed to identify those quandle cocycles that correspond to quandle coverings.

\begin{definition}[{{\cite[Definition 2.2]{AG}}}]\label{Df:CQC}
Let $\X$ be a quandle and $S$ a nonempty set. A quandle cocycle $\beta\in Z^2(X,\sym{S})$ is a \emph{constant quandle cocycle} if
\begin{displaymath}
    \beta(x,y,r)=\beta(x,y,s)
\end{displaymath}
for every $x$, $y\in X$ and $r$, $s\in S$. Since the value of $\beta(x,y,s)$ is then independent of $s\in S$, we will think of constant cocycles as mappings $\beta:X\times X\to \sym{S}$.

The set of all constant quandle cocycles $X\times X\to \sym{S}$ will be denoted by $Z^2_c(X,\sym{S})$. The equivalence relation $\sim$ on $Z^2(X,\sym{S})$ induces an equivalence relation on $Z^2_c(X,\sym{S})$, and we define
\begin{displaymath}
    H^2_c(\X,\sym{S}) =Z^2_c(\X,\sym{S})/\sim,
\end{displaymath}
the \emph{second constant cohomology set} of $X$ with coefficients in $\sym{S}$.
\end{definition}

We see immediately from \eqref{cocycle condition} and \eqref{quandle condition} that a mapping $\beta:X\times X\to\sym{S}$ is a constant quandle cocycle if and only if it satisfies
\begin{align}
    \beta ( xy,xz)\beta ( x,z)&=\beta (x,yz) \beta ( y,z)\tag{CC}\label{CCC},\\
    \beta ( x,x)&=1\tag{CQ}\label{CQC}
\end{align}
for every $x$, $y$, $z\in X$.

Note that \eqref{CCC} implies
\begin{equation}\label{WCC}
    \beta (xy,xz) =\beta(x,yz)\ \Leftrightarrow\ \beta(x,z) = \beta(y,z)\tag{WCC}
\end{equation}
for every $x$, $y$, $z\in X$. We will call \eqref{WCC} the \emph{weaker cocycle condition}.

Just as quandle cocycles parametrize quandle extensions, the constant cocycles parametrize quandle coverings.

\begin{proposition}[{{\cite[Lemma 5.1]{CEGS}}}]\label{coverings and constant cocycles}
Let $X$, $Y$ be connected quandles. Then the following conditions are equivalent:
\begin{enumerate}
\item[(i)] $Y$ is a covering of $X$,
\item[(ii)] $Y\cong X\times_\beta S$ for some set $S$ and $\beta\in Z^2_c(X,\sym{S})$.
\end{enumerate}
\end{proposition}
\begin{proof}
Let $Y$ be an extension of $X$, say $Y=X\times_\beta S$ for $\beta\in Z^2(X,\sym{S})$. Then $(x,r)\cdot (y,t)=(x,s)\cdot(y,t)$ for every $x$, $y\in X$, $r$, $s$, $t\in S$ if and only if $\beta(x,y,r)=\beta(x,y,s)$ for every $x$, $y\in X$, $r$, $s\in S$, which is equivalent to $\beta\in Z^2_c(X,\sym{S})$.
\end{proof}

Let $X$ be a quandle and $S$ a nonempty set. The mapping defined by
\begin{displaymath}
    X\times X\longrightarrow \sym{S}, \quad (x,y) \mapsto 1
\end{displaymath}
is a constant cocycle, called the \emph{trivial cocycle} and denoted by $1$. It is easy to see that $X\times_1 S$ is the direct product of $X$ and the projection quandle over $S$. The covering $Y=X \times_1 S$ is called a \emph{trivial covering} of $X$.


Two coverings $f:Y\to X$, $f':Y'\to X$ of $X$ are said to be \emph{equivalent} if there is a quandle isomorphism $\phi:Y\to Y'$ such that $f'\circ \phi = f$.

Let $X=(X,\cdot)$ be a quandle. The \emph{adjoint group} $\mathrm{Adj}(X)$ of $X$ is the group with generators $\setof{e_x}{x\in X}$ and presenting relations $\setof{e_{x\cdot y} = e_x^{-1}e_ye_x}{x,\,y\in X}$. Following \cite[Definitions 1.7, 1.10]{E}, let $\epsilon:\mathrm{Adj}(X)\to \mathbb Z$ be the unique homomorphism such that $\epsilon(e_x)=1$ for every $x\in X$. Let $\mathrm{Adj}(X)^\circ$ be the kernel of $\epsilon$. The \emph{fundamental group} of $X$ based at $x\in X$ is defined as $\pi_1(X,x) = \setof{g\in\mathrm{Adj}(X)^\circ}{x^g=x}$. By \cite[Proposition 5.8]{E}, $\pi_1(X,x)$ is conjugate to $\pi_1(X,y)$ whenever $x$, $y$ are in the same orbit of $\lmlt{X}$. In particular, if $X$ is a connected quandle then the isomorphism type of $\pi_1(X,x)$ is independent of the base point $x$, and it is safe to write $\pi_1(X)$ instead of $\pi_1(X,x)$.

\begin{proposition}\label{Pr:SC}
The following conditions are equivalent for a connected quandle $X$:
\begin{enumerate}
\item[(i)] $\pi_1(X)$ is trivial,
\item[(ii)] every covering $Y\to X$ is equivalent to the trivial covering of $X$ over some set $S$,
\item[(iii)] $H^2_c(X,\sym{S})$ is trivial for every set $S$.
\end{enumerate}
\end{proposition}
\begin{proof}
The equivalence of (i) and (ii) is established in \cite[Proposition 5.15]{E}. Let us prove the equivalence of (ii) and (iii).

By Proposition \ref{coverings and constant cocycles}, any covering of $X$ is of the form $\pi:X\times_\beta S\to S$ for some nonempty set $S$. If $X\times_\beta S\to S$, $X\times_{\beta'} S'\to S'$ are two equivalent coverings of $X$, then $S$ and $S'$ have the same cardinality. It therefore suffices to investigate two coverings $X\times_\beta S\to S$ and $X\times_{\beta'} S\to S$ with $\beta$, $\beta'\in Z^2_c(X,\sym{S})$. By Proposition \ref{Pr:Categorical}, these two coverings are equivalent if and only if $\beta\sim\beta'$.
\end{proof}

\section{Normalized constant cocycles}\label{Sc:NormalizedCocycles}

In this section we start computing the constant cohomology set $H^2_c(X,\sym{S})$ of a latin quandle $X$. The situation is greatly simplified in the latin case because every cocycle of $H^2_c(X,\sym{S})$ can be assumed to be normalized:

\begin{definition}[{{compare \cite[Lemma 5.1]{G}}}]
Let $\X$ be a latin quandle, $S$ a nonempty set and $u\in X$. Then $\beta\in Z^2_c(X,\sym{S})$ is said to be \emph{$u$-normalized} if
\begin{displaymath}
    \beta(x,u) = 1
\end{displaymath}
for every $x\in X$.
\end{definition}

For $\beta\in Z^2_c(X,\sym{S})$ and $\sigma\in\sym{S}$ define $\beta^\sigma\in Z^2_c(X,\sym{S})$ by $\beta^\sigma(x,y) = \sigma\beta(x,y)\sigma^{-1}$.

\begin{proposition}\label{Prop:cocycle normalized}
Let $\X$ be a latin quandle, $S$ a nonempty set and $u\in X$. For $\beta\in Z^2_c(X,\sym{S})$ define $\beta_u\in Z^2_c(X,\sym{S})$ by
\begin{displaymath}
    \beta_u(x,y) = \beta((xy)/u,u)^{-1}\beta(x,y)\beta(y/u,u).
\end{displaymath}
Then $\setof{\beta_u^\sigma}{\sigma\in\sym{S}}$ is the set of all $u$-normalized cocycles cohomologous to $\beta$.
\end{proposition}
\begin{proof}
We have $\delta\sim\beta$ if and only if $\delta(x,y)=\gamma(xy)\beta(x,y)\gamma(y)^{-1}$ for some $\gamma:X\to\sym{S}$. The following conditions are then equivalent for $\delta$: $\delta$ is $u$-normalized, $1=\delta(x,u) = \gamma(xu)\beta(x,u)\gamma(u)^{-1}$ for every $x\in X$, $\gamma(xu) = \gamma(u)\beta(x,u)^{-1}$ for every $x\in X$,
\begin{equation}\label{Eq:Gamma}
    \gamma(x) = \gamma(u)\beta(x/u,u)^{-1}
\end{equation}
for every $x\in X$, where we have used the latin property in the last step. Conversely, given $\sigma=\gamma(u)\in\sym{S}$, the formula \eqref{Eq:Gamma} defines a map $\gamma:X\to\sym{S}$ (it is well defined since $\gamma(u)\beta(u/u,u)^{-1} = \gamma(u)\beta(u,u)^{-1}=\gamma(u)$ by \eqref{CQC}). Then
\begin{displaymath}
    \delta(x,y){=}\gamma(xy)\beta(x,y)\gamma(y)^{-1}{=}\sigma\beta((xy)/u,u)^{-1}\beta(x,y)\beta(y/u,u)\sigma^{-1}{=}\sigma\beta_u(x,y)\sigma^{-1},
\end{displaymath}
so $\delta = \beta_u^\sigma$.
\end{proof}

\begin{corollary}\label{trivial cohomology and normalized}
Let $\X$ be a latin quandle, $S$ a nonempty set, and $u\in X$. Let $\beta$, $\beta'\in Z^2_c(X,\sym{S})$, and let $\delta$, $\delta'$ be $u$-normalized cocycles such that $\delta\sim\beta$ and $\delta'\sim\beta'$. Then $\beta\sim\beta'$ if and only if $\delta'=\delta^\sigma$ for some $\sigma\in\sym{S}$. Moreover, the following conditions are equivalent:
\begin{enumerate}
\item[(i)] $H^2_c(X,\sym{S})$ is trivial,
\item[(ii)] if $\beta\in Z^2_c(X,\sym{S})$, $\delta\sim\beta$ and $\delta$ is $u$-normalized then $\delta=1$,
\item[(iii)] $\beta_u=1$ for every $\beta\in Z^2_c(X,\sym{S})$.
\end{enumerate}
\end{corollary}
\begin{proof}
The first statement follows immediately from Proposition \ref{Prop:cocycle normalized}. Suppose that (i) holds, let $\beta\in Z^2_c(X,\sym{S})$, and let $\delta\sim\beta$ be $u$-normalized. Since $1$ is also $u$-normalized and $1\sim\beta$ by triviality of $H^2_c(X,\sym{S})$, we have $\delta=1$ by the first statement, establishing (ii). Clearly, (ii) implies (iii). Finally, if (iii) holds then $\beta\sim\beta_u=1$ for every $\beta\in Z^2_c(X,\sym{S})$, so $H^2_c(X,\sym{S})$ is trivial.
\end{proof}

Many identities for normalized cocycles can be derived from \eqref{CCC} and \eqref{CQC}. We will later need:

\begin{lemma}\label{some prop of beta}
Let $\X$ be a latin quandle and let $\beta $ be a $u$-normalized cocycle. Then
\begin{displaymath}
\beta \left( u/\left( u/x\right) ,x\right) =\beta \left( u/x,x\right)
\end{displaymath}
for every $x\in X$. Moreover $u/(u/x)\cdot x=u$ if and only if $x=u$.
\end{lemma}

\begin{proof}
Setting $x=u/y$ and $y=u/z$ in \eqref{CCC}, we get $\beta( u/ ( u/z) ,z) =\beta  ( u/z,z)$ for every $z\in X$. Moreover,
\begin{displaymath}
u/(u/x)\cdot x=u \ \Leftrightarrow \ u/(u/x)=u/x \ \Leftrightarrow \ u/x=u \ \Leftrightarrow \ x=u.
\end{displaymath}
\end{proof}

\section{Three bijections on $X\times X$ that preserve normalized cocycles}\label{Sc:ThreeBijections}

Given a mapping $\alpha:A\to B$ and a bijection $\ell:A\to A$, we say that $\alpha$ is \emph{$\ell$-invariant} if $\alpha(x) = \alpha(\ell(x))$ for every $x\in A$.

In this section we introduce three bijections of $X\times X$ (where $X$ is a latin quandle) under which all normalized cocycles are invariant. We will use these bijections throughout the rest of the paper.

Let $X$ be a latin quandle and $u\in X$. Define
\begin{align}
    &f:X\times X\to X\times X,\quad (x,y)\mapsto (x\cdot y/u,xu),\notag\\
    &g:X\times X\to X\times X,\quad (x,y)\mapsto (ux,uy),\label{Eq:fgh}\\
    &h:X\times X\to X\times X,\quad (x,y)\mapsto (y/(x\backslash u)\cdot x,y).\notag
\end{align}
The element $u$ on which $f$, $g$, $h$ depend will always be understood from the context.

\begin{proposition}\label{Prop:invariance under f}
Let $X$ be a latin quandle and $u\in X$. Then $f\in\sym{X\times X}$ and every $u$-normalized cocycle is $f$-invariant.
\end{proposition}
\begin{proof}
It is easy to see that $f$ has an inverse, namely the mapping $(x,y)\mapsto (y/u,(y/u)\ldiv x\cdot u)$. Let $\beta$ be a $u$-normalized cocycle. Then \eqref{WCC} implies $\beta(xy,xu) = \beta(x,yu)$ for every $x$, $y\in X$. With $z=yu$, we obtain $\beta(x,z) = \beta(x,yu) = \beta(xy,xu) = \beta(x(z/u),xu) = \beta(f(x,z))$.
\end{proof}

\begin{lemma}\label{Lm:LatinEq}
Let $X$ be a latin quandle, $u\in X$ and $\beta$ a constant cocycle. Then the following conditions are equivalent:
\begin{enumerate}
\item[(i)] $\beta(u,x) =1 $ for every $x\in X$,
\item[(ii)] $\beta$ is $g$-invariant.
\end{enumerate}
\end{lemma}
\begin{proof}
If (i) holds, then
\begin{displaymath}
    \beta ( ux,uy ) \overset{\text{(i)}}{=}
    \beta ( ux,uy ) \beta  ( u,y ) \overset{\eqref{CCC}}{=}
    \beta ( u,xy ) \beta  ( x,y ) \overset{\text{(i)} }{=}
    \beta ( x,y )
\end{displaymath}
for every $x$, $y\in X$. Conversely, if (ii) holds, let $x=u/y$ and verify
\begin{align*}
    \beta ( u,y )&\overset{\eqref{CCC}}{=} \beta ( u\cdot u/y ,uy )^{-1}\beta ( u,  u/y\cdot y) \beta ( u/y,y)\\
    &\overset{\phantom{\eqref{CCC}}}{=} \beta ( u\cdot  u/y ,uy ) ^{-1}\beta  ( u,u )\beta  ( u/y,y )\\
    &\overset{\eqref{CQC} }{=} \beta ( u\cdot u/y ,uy ) ^{-1}\beta  ( u/y,y ) \overset{\text{(ii)}}{=} 1
\end{align*}
for every $y\in X$.
\end{proof}

We remark that (i) implies (ii) in Lemma \ref{Lm:LatinEq} even if $X$ is an arbitrary quandle, not necessarily latin.

\begin{proposition}\label{Prop:invariance_under_g}
Let $\X$ be a latin quandle and $u\in X$. Then $g\in\sym{X\times X}$ and every $u$-normalized cocycle is $g$-invariant.
\end{proposition}
\begin{proof}
Since $g=L_u\times L_u$, we obviously have $g\in\sym{X\times X}$ for any quandle $X$. Suppose now that $\beta$ is a $u$-normalized cocycle. In view of Lemma \ref{Lm:LatinEq}, it suffices to prove that $\beta(u,x)=1$ for every $x\in X$. Now,
\begin{displaymath}
    \beta(u,xu)=\beta(u,xu)\beta(x,u)\overset{\eqref{CCC}}{=}\beta(ux,u)\beta(u,u)=1,
\end{displaymath}
and we are done because $R_u$ is a bijection.
\end{proof}

\begin{lemma}
The following identities hold in a latin quandle:
\begin{align}
    (xy)/z &= x(y/(x\ldiv z)),\label{Eq:LID2}\\
    (x/y)(zy) &= ((x/y)z)x.\label{Eq:LID3}
\end{align}
\end{lemma}
\begin{proof}
For \eqref{Eq:LID2}, substitute $x\ldiv z$ for $z$ in $x(y/z) = (xy)/(xz)$. The identity \eqref{Eq:LID3} follows immediately from left distributivity and $(x/y)y=x$.
\end{proof}

\begin{proposition}\label{Prop:invariance_under_h}
\label{Prop:invariance under left action}Let $\X$ be a latin quandle and $u\in X$. Then $h\in\sym{X\times X}$ and every $u$-normalized cocycle is $h$-invariant.
\end{proposition}
\begin{proof}
We will show that
\begin{displaymath}
    k(x,y) = (u/((xy/u)\ldiv y),y)
\end{displaymath}
is the inverse of $h$. (The mapping $k$ was found by automated deduction \cite{McCune}.) It suffices to show that $h$, $k$ are inverse to each other in the first coordinate. We will freely use the quasigroup identities $x/(y\ldiv x)=y$ and $(x/y)\ldiv x = y$.

The first coordinate of $h(k(x,y))$ is equal to
\begin{align*}
    &y/[ (u/((xy/u)\ldiv y))\ldiv u]\cdot u/((xy/u)\ldiv y) \\
    &=y/[(xy/u)\ldiv y]\cdot u/((xy/u)\ldiv y)\\
    &=(xy/u)\cdot u/((xy/u)\ldiv y),
\end{align*}
which is an expression of the form $a\cdot b/(a\ldiv c)$ and therefore, by \eqref{Eq:LID2}, equal to $((xy/u)u)/y = xy/y = x$.

The first coordinate of $k(h(x,y))$ is equal to
\begin{equation}\label{Eq:Aux1}
    u/ [(((y/(x\ldiv u)\cdot x)y)/u)\ldiv y].
\end{equation}
Since $(y/(x\ldiv u)\cdot x)y$ is of the form $(a/b)c\cdot a$, it is by \eqref{Eq:LID3} equal to $(y/(x\ldiv u))(x\cdot x\ldiv u) = y/(x\ldiv u)\cdot u$, and substituting this back into \eqref{Eq:Aux1} yields
\begin{displaymath}
    u/ [((y/(x\ldiv u)\cdot u)/u)\ldiv y] = u/[(y/(x\ldiv u))\ldiv y] = u/(x\ldiv u) = x.
\end{displaymath}
We have proved that $k$ is the inverse of $h$.

Let $\beta$ be a $u$-normalized cocycle. Then we have
\begin{align*}
    \beta(u/y,y)&=\beta(u/y\cdot x,u)\beta(u/y,y)\\
    &=\beta(u/y\cdot x,u/y\cdot y)\beta(u/y,y)\overset{\eqref{CCC}}{=} \beta(u/y,xy)\beta(x,y),
\end{align*}
and also
\begin{displaymath}
    \beta(u/y,y) = \beta(x,u)\beta(u/y,y) = \beta(x,u/y\cdot y)\beta(u/y,y)
    \overset{\eqref{CCC}}{=} \beta(x\cdot u/y,xy)\beta(x,y).
\end{displaymath}
Therefore $\beta(u/y,xy) = \beta(x\cdot u/y,xy)$, and with $u/y=z$, $xy=v$ we have $x=v/y = v/(z\ldiv u)$ and $\beta(z,v) = \beta(v/(z\ldiv u)\cdot z,v) = \beta(h(z,v))$.
\end{proof}

\section{Orbits of the three bijections on $X\times X$}\label{Sc:Orbits}

For $\ell\in\sym{A}$ and $a\in A$ let $O_\ell(a)$ be the orbit of $a$ under the natural action of $\langle\ell\rangle$ on $A$, and let $\mathcal O_\ell=\setof{O_\ell(a)}{a\in A}$ be the collection of all $\ell$-orbits.

In this section we study the orbits of the bijections $f$, $g$, $h$ on $X\times X$ defined in \eqref{Eq:fgh}, and also the orbits of the induced action of $\langle f,h\rangle$ on $\mathcal O_g$.

Let $X$ be a quandle. Denote by $p$ the \emph{product mapping}
\begin{displaymath}
    p:X\times X\to X,\quad (x,y)\mapsto xy.
\end{displaymath}
For $z\in X$, the fiber $p^{-1}(z)$ is equal to
\begin{displaymath}
    p^{-1}(z) = \setof{(x,y)\in X\times X}{xy=z} = \setof{(x,x\ldiv z)}{x\in X}
\end{displaymath}
and hence has cardinality $|X|$. Moreover, since $p(f(x,y)) = p(x\cdot y/u,xu) = (x\cdot y/u)(xu) = x(y/u\cdot u) = xy = p(x,y)$, every fiber is a union of $f$-orbits.

We have $O_g(u,u)=\{(u,u)\}$, and we collect some additional orbits of $g$ by setting
\begin{align*}
    \mathcal O_g^f &= \setof{O_g(x,xu)}{u\ne x\in X},\\
    \mathcal O_g^u &= \setof{O_g(x,x\ldiv u)}{u\neq x\in X}.
\end{align*}
The notation is explained as follows. By Lemma \ref{Lm:some orbits of g} below, $\bigcup \mathcal O_g^f=\setof{(x,xu)}{u\ne x\in X}$ and $\bigcup \mathcal O_g^u=\setof{(x,x\ldiv u)}{u\ne x\in X}$. By Proposition \ref{Prop:orbit sizes of f} then, $\{(u,u)\}\cup \bigcup\mathcal O_g^f$ are precisely the fixed points of $f$, while $\{(u,u)\}\cup \bigcup \mathcal O_g^u$ is the fiber $p^{-1}(u)$.

We will ultimately prove that certain quandles are simply connected by the following strategy, exploiting the invariance under $f$, $g$ and $h$ of $u$-normalized cocycles (see Propositions \ref{Prop:invariance under f}, \ref{Prop:invariance_under_g} and \ref{Prop:invariance under left action}). For a $u$-normalized cocycle $\beta$, we first partition its domain $X\times X$ into $g$-orbits $\mathcal O_g$ on which $\langle f,h\rangle$ acts by Proposition \ref{Pr:InducedAction}. By Corollary \ref{action of omega on diag}, $f$ acts on both $\mathcal O_g^u$ and $\mathcal O_g^f$, while $h$ most definitely does not.  The bijection $h$ is much easier to understand in affine connected quandles, cf. Lemma \ref{geometric prog}. The affine-over-cyclic case of Theorem \ref{Th:Main} then easily follows. In the doubly transitive case, we will show that there are at most five orbits of $\langle f,g,h\rangle$, namely $O_g(u,u)$, $\bigcup\mathcal O_g^u$, $\bigcup\mathcal O_g^f$ and certain sets $\bigcup\mathcal O_g^1$, $\bigcup\mathcal O_g^2$ introduced later. (We note that $\mathcal O_g^f$, $\mathcal O_g^u$ need not be disjoint, but their intersection is easy to understand, cf. Lemma \ref{Lm:some orbits of g}.) A careful analysis of the orbit sizes of $h$ then shows that the four sets $\mathcal O_g^u$, $\mathcal O_g^f$, $\mathcal O_g^1$ and $\mathcal O_g^2$ must be linked by $h$ (as long as $|X|\ne 4$), establishing the main result.

\begin{lemma}\label{f orbits}
Let $\X$ be a latin quandle. Then for every $x$, $y\in X$ and every $k\in\mathbb Z$ we have
\begin{equation*}
    f^k(x,y) =(f_k(x,y),f_{k-1}(x,y)u),
\end{equation*}
where
\begin{equation}\label{fks}
    f_k(x,y) = \left\{
    \begin{array}{l}
        (L_{x} L_{y/u})^{\frac{k}{2}}(x),\text{ if k is even,}  \\
        (L_{x} L_{y/u})^{\frac{k+1}{2}}(y/u),\text{ if k is odd.}
    \end{array}\right.
\end{equation}
\end{lemma}
\begin{proof}
Fix $x$, $y\in X$. Let $\varphi = L_xL_{y/u}$ and define $f_k$, $f'_k$ by $f^k(x,y) = (f_k,f'_k)$. Then $(f_{k+1},f'_{k+1}) = f(f_k,f'_k) = (f_k\cdot f_k/u,f_ku)$, so $f'_{k+1}=f_ku$ and $f_{k+1} = f_k\cdot f'_k/u = f_kf_{k-1}$ for every $k$. Note that we have $\varphi^{k+1}(y/u) = \varphi^k(\varphi(y/u)) = \varphi^k(x(y/u\cdot y/u)) = \varphi^k(x\cdot y/u)$.

For the base step, we will show that \eqref{fks} holds for $k\in\{0,1\}$. Indeed, $f_0 = x = \varphi^0(x)$ and $f_1 = x\cdot y/u = x(y/u\cdot y/u) = \varphi(y/u)$.

For the ascending induction, suppose that \eqref{fks} holds for $k-1$ and $k$. If $k$ is even, we have
\begin{displaymath}
    f_{k+1} = f_kf_{k-1} = \varphi^{k/2}(x)\varphi^{k/2}(y/u) = \varphi^{k/2}(x\cdot y/u) = \varphi^{(k+2)/2}(y/u).
\end{displaymath}
If $k$ is odd, we have
\begin{align*}
    f_{k+1} &= f_kf_{k-1} = \varphi^{(k+1)/2}(y/u)\varphi^{(k-1)/2}(x) = \varphi^{(k-1)/2}(x\cdot y/u)\varphi^{(k-1)/2}(x)\\
    &= \varphi^{(k-1)/2}((x\cdot y/u)x) = \varphi^{(k-1)/2}(x(y/u\cdot x)) = \varphi^{(k+1)/2}(x),
\end{align*}
where we have used flexibility.

For the descending induction, suppose that \eqref{fks} holds for $k$ and $k+1$. If $k$ is even then
\begin{displaymath}
    f_{k-1} = f_k\ldiv f_{k+1} = \varphi^{k/2}(x)\ldiv \varphi^{k/2}(x\cdot y/u) = \varphi^{k/2}(x\ldiv (x\cdot y/u)) = \varphi^{k/2}(y/u).
\end{displaymath}
If $k$ is odd then
\begin{align*}
    f_{k-1} &= f_k\ldiv f_{k+1} = \varphi^{(k+1)/2}(y/u)\ldiv \varphi^{(k+1)/2}(x)\\
     &= \varphi^{(k+1)/2}((y/u)\ldiv x) = \varphi^{(k-1)/2}\varphi((y/u)\ldiv x) = \varphi^{(k-1)/2}(x),
\end{align*}
finishing the proof.
\end{proof}

\begin{proposition}\label{Prop:orbit sizes of f}
Let $\X$ be a latin quandle and $x$, $y\in X$. Then, using the notation of Lemma \emph{\ref{f orbits}}, the following conditions are equivalent: $f^k(x,y)=(x,y)$, $f_k(x,y)=x$, $f_{k-1}(x,y)=y/u$. In particular,
\begin{enumerate}
\item[(i)] $|O_f(x,y)|=1$ if and only if $y=xu$,
\item[(ii)] $|O_f(x,y)|\ne 2$,
\item[(iii)] $|O_f(x,y)|\le |X|$.
\end{enumerate}
\end{proposition}
\begin{proof}
Fix $x$, $y\in X$ and let $f_k=f_k(x,y)$. Clearly, $f^k(x,y)=(x,y)$ holds if and only if both $f_k=x$ and $f_{k-1}=y/u$ hold. Since $xy = p(x,y) = p(f^k(x,y)) = p(f_k,f_{k-1}u) =  f_k\cdot f_{k-1}u$, we have $f_k=x$ if and only if $f_{k-1}u=y$.

Part (i) now follows. Suppose that $f^2(x,y) = (x,y)$. The equality $x=f_2$ says $x = x\cdot (y/u\cdot x)$, which implies $x=x\cdot y/u$, $x=y/u$, $y=xu$. But then $|O_f(x,y)|=1$ by (i). Finally, (iii) follows from the above-mentioned fact that every fiber of $p$ has cardinality $|X|$ and is a union of $f$-orbits.
\end{proof}

\begin{proposition}\label{Prop:orbit sizes of g}
Let $\X$ be a finite latin quandle and $x$, $y\in X$. Then $|O_g(x,y)| = \lcm(|O_{L_u}(x)|,|O_{L_u}(y)|)$. Moreover, $|O_g(x,y)|=1$ if and only if $(x,y)=(u,u)$.
\end{proposition}
\begin{proof}
This follows immediately from $g^k(x,y) = (L_u^k(x),L_u^k(y))$.
\end{proof}

\begin{lemma}\label{Lm:fixed pts of h}
Let $\X$ be a latin quandle and $x$, $y\in X$. Then the following conditions are equivalent:
\begin{enumerate}
\item[(i)] $|O_h(x,y)|=1$,
\item[(ii)] $p(h(x,y)) = p(x,y)$,
\item[(iii)] $y=u$.
\end{enumerate}
\end{lemma}
\begin{proof}
Each of (i), (ii) is equivalent to $y/(x\ldiv u)\cdot x = x$, which is equivalent to $y/(x\ldiv u) = x$, that is, to $y=u$.
\end{proof}

The action of $\langle f,h\rangle$ on $X\times X$ in fact induces an action on $\mathcal O_g$, the orbits of $g$:

\begin{proposition}\label{Pr:InducedAction}
Let $\X$ be a latin quandle. Then
\begin{displaymath}
     k(O_g(x,y)) = O_g(k(x,y))
\end{displaymath}
for every $k\in \langle f, h\rangle$.
\end{proposition}
\begin{proof}
It suffices to show that $f$ and $h$ commute with $g$. Let $x$, $y\in X$. Since $L_u\in\aut{X}$, we have
\begin{displaymath}
    f(g(x,y)) = f(ux,uy) = (ux\cdot (uy)/u ,uxu) = ( u\cdot (x\cdot (y/u)),u\cdot xu) = g(f(x,y))
\end{displaymath}
and
\begin{displaymath}
    h(g(x,y)){=}h(ux,uy){=}(uy/(ux\ldiv u)\cdot ux, uy){=}(u\cdot (y/(x\ldiv u)\cdot x),u\cdot y){=}g(h(x,y)).
\end{displaymath}
\end{proof}

\begin{remark}
The mappings $f$ and $h$ never commute on a nontrivial latin quandle. More precisely, $fh(x,y) = hf(x,y)$ if and only if $x=y=u$. Indeed, we certainly have $f(u,u) = (u,u) = h(u,u)$, so $fh(u,u) = hf(u,u)$. Suppose now that
\begin{displaymath}
    fh(x,y) =f(y/(x\ldiv u)\cdot x,y) = ((y/(x\ldiv u)\cdot x)(y/u),(y/(x\ldiv u)\cdot x)u)
\end{displaymath}
is equal to
\begin{displaymath}
    hf(x,y) = h(x\cdot y/u,xu) = (xu/((x\cdot y\ldiv u)\ldiv u)\cdot x(y\ldiv u),xu).
\end{displaymath}
By comparing the second coordinates we see that $y=u$, and substituting this into the first coordinates we arrive at $xu = xu/(xu\ldiv u)\cdot xu$, that is, $x=u$.
\end{remark}

\begin{lemma}\label{Lm:some orbits of g}
Let $X$ be a latin quandle and $x\in X$. Then
\begin{itemize}
\item[(i)] $O_g(x,xu)=\setof{(y,yu)}{y\in O_{L_u}(x)}$,
\item[(ii)] $O_g(x,x\ldiv u)=\setof{(y,y \ldiv u)}{y\in O_{L_u}(x)}$.
\end{itemize}
In particular, $O_g(x,y)\in \mathcal O_g^u\cap \mathcal O_g^f$ if and only if $y=xu$ and $x\cdot xu=u$.
\end{lemma}
\begin{proof}
We have $g(x,xu) =  (ux,uxu)$, $g^{-1}(x,xu) = (u\ldiv x, u\ldiv (xu)) = (u\ldiv x, (u\ldiv x)u)$, and (i) follows by induction. Similarly, $g(x,x\ldiv u) = (ux,u\cdot x\ldiv u) = (ux,(ux)\ldiv u)$ and $g^{-1}(x,x\ldiv u) = (u\ldiv x,u\ldiv (x\ldiv u) = (u\ldiv x,(u\ldiv x)\ldiv u)$ prove (ii). Hence $O_g(x,y)\in \mathcal O_g^u\cap \mathcal O_g^f$ if and only if $(x,y)=(z,zu)=(w,w\ldiv u)$ for some $z$, $w\in X$, which is equivalent to $z=w=x$, $y=xu=x\ldiv u$.
\end{proof}

\begin{corollary}\label{action of omega on diag}
Let $X$ be a latin quandle. Then
\begin{enumerate}
\item[(i)] $h(\mathcal O_g^u)\cap \mathcal O_g^u = \emptyset$,
\item[(ii)] $h(\mathcal O_g^f)\cap \mathcal O_g^f = \emptyset$,
\item[(iii)] $f(O_g(x,xu))=O_g(x,xu)$ for every $x\in X$, in particular, $f(\mathcal O_g^f) = \mathcal O_g^f$,
\item[(iv)] $f(\mathcal O_g^u)= \mathcal O_g^u$.
\end{enumerate}
\end{corollary}
\begin{proof}
For (i), suppose that $h(O_g(x,x\ldiv u))\in \mathcal O_g^u$ for some $x\ne u$. By Lemma \ref{Lm:some orbits of g}, $h(x,x\ldiv u)=(y,y\ldiv u)$ for some $y\in X$. But then $x=y$, $h(x,x\ldiv u)=(x,x\ldiv u)$, and thus $x=u$ by Lemma \ref{Lm:fixed pts of h}, a contradiction.

The proof of (ii) is similar: If $h(O_g(x,xu))\in \mathcal O_g^f$ for some $x\ne u$, then $h(x,xu)=(y,yu)$ for some $y$ by Lemma \ref{Lm:some orbits of g}, hence $x=y$, $h(x,xu)=(x,xu)$, $xu=u$, $x=u$, a contradiction.

For (iii), recall that $f(x,xu) = (x,xu)$ by Proposition \ref{Prop:orbit sizes of f}.

Finally, for (iv), note that $p(f(x,x\ldiv u))=p(x,x\ldiv u)=u$, hence $f(x,x\ldiv u)$ must be equal to some $(y,y\ldiv u)$. Moreover, when $x\ne u$ then $y\ne u$ because $f$ fixes $(u,u)$.
\end{proof}

A permutation is said to be \emph{semiregular} if all its nontrivial cycles are of the same finite length. A quandle $X$ is \emph{semiregular} if there is a positive integer $s$ such that every nontrivial cycle of any left translation $L_x$ has length $s$.

Clearly, if $X$ is a connected quandle and $L_u$ is semiregular for some $u\in X$, then $X$ is semiregular. In particular, a latin quandle is semiregular if and only if $L_u$ is semiregular for some $u\in X$.

Let us denote a typical orbit of the induced action of $f$ on $\mathcal O_g$ by $O_f(O_g(x,y))$, cf. Proposition \ref{Pr:InducedAction}. Then $|O_f(O_g(x,y))|\le |O_f(x,y)|$ and the strict inequality can occur in examples.

\begin{lemma}\label{F on Delta}
Let $X$ be a latin semiregular quandle. Then $|O_f(O_g(x,y))| = |O_f(x,y)|$ and $|O_h(O_g(x,y))| = |O_h(x,y)|$ for every $x$, $y\in X$. Hence $f(O_g(x,y))=O_g(x,y)$ if and only if $O_g(x,y)\in \mathcal O_g^f$.
\end{lemma}
\begin{proof}
Suppose that $f^k(O_g(x,y))=O_g(x,y)$ and $k>0$ is smallest possible. Then $f^k(x,y)=(L_u^r(x),L_u^r(y))$ for some $r\in\mathbb Z$. Therefore, $xy = p(x,y) = p(L_u^r(x),L_u^r(y)) = L_u^r(xy)$. But then semiregularity implies that $|L_u|$ divides $r$, $f^k(x,y) = (x,y)$, and $|O_f(O_g(x,y))|\ge |O_f(x,y)|$.

Similarly, suppose that $h^k(O_g(x,y))=O_g(x,y)$ and $k>0$ is smallest possible. Then $h^k(x,y)=(z,y)=(L_u^r(x),L_u^r(y))$ for some $r\in\mathbb Z$ and some $z\in X$. Then $L_u^r(y)=y$, hence $L_u^r(x)=x$ by semiregularity, $h^k(x,y)=(x,y)$, and $|O_h(O_g(x,y))|\ge |O_h(x,y)|$.
\end{proof}

\section{The orbits on connected affine quandles}\label{Sc:OrbitsOnAffine}

In this section we take advantage of the affine representation to arrive at explicit expressions for the mappings $f$ and $h$ in terms of the underlying group and the automorphism $\alpha$. Moreover, we compute the orbit lengths for $f$ and $h$.

We will use additive notation for the underlying groups and set $u=0$ in affine quandles. We therefore also write $\mathcal O_g^0$ for $\mathcal O_g^u$.

The results from previous sections apply to finite affine connected quandles thanks to the following, well-known result.

\begin{proposition}[{{\cite[Proposition 1]{Lit}}}]\label{connected iff latin}
Let $X=\gal{A,\alpha}$ be a finite affine quandle. Then the following conditions are equivalent:
\begin{enumerate}
\item[(i)] $X$ is latin,
\item[(ii)] $X$ is connected,
\item[(iii)] $1-\alpha\in \aut{A}$.
\end{enumerate}
\end{proposition}
%
%

Note that Proposition \ref{connected iff latin} implies that in any finite connected affine quandle $X=\gal{A,\alpha}$ we have $L_0(y) = (1-\alpha)(0)+\alpha(y) = \alpha(y)$, that is, $L_0=\alpha$.

\begin{proposition}\label{iterative multiplication}
Let $A$ be an abelian group and let $\alpha\in\aut{A}$ be such that $1-\alpha \in \aut{A}$. Then for every $x\in A$ and every positive integer $n$ we have
\begin{displaymath}
    \alpha^n(x) =x\ \Leftrightarrow\ \sum_{k=0}^{n-1}\alpha^k(x) =0.
\end{displaymath}
\end{proposition}
\begin{proof}
We have $1-\alpha^n = (1-\alpha)\sum_{k=0}^{n-1}\alpha^k$. Since $1-\alpha\in \aut{A}$, we deduce that $\alpha^n(x)=x$ if and only if $\sum_{k=0}^{n-1}\alpha^k(x)=0$.
\end{proof}

\begin{lemma}\label{Lm:iterated product}
Let $X=\gal{A,\alpha}$ be an affine quandle. Then for every $x$, $y\in X$ and $k\ge 0$, the element $f_k(x,y)$ from \eqref{fks} is equal to
\begin{equation}\label{fks again}
    f_k(x,y)= x+\sum_{j=1}^k (-1)^j\alpha^j(x-y/0).
\end{equation}
\end{lemma}
\begin{proof}
Fix $x$, $y\in X$ and let $f_k = f_k(x,y)$. The formula \eqref{fks} yields $f_0=x$ and $f_1 = x\cdot (y/0) = x-\alpha(x)+\alpha(y/0)$, in agreement with \eqref{fks again}. Suppose that \eqref{fks again} holds for $k-1$ and $k$, and recall that $f_{k+1} = f_kf_{k-1}$. Then with $z=x-y/0$ we have
\begin{align*}
    f_{k+1} &= (x+\sum_{j=1}^k (-1)^j\alpha^j(z))\cdot(x+\sum_{j=1}^{k-1}(-1)^j\alpha^j(z))\\
        &= x - \alpha(x) + \sum_{j=1}^k (-1)^j\alpha^j(z) -  \sum_{j=1}^k (-1)^j\alpha^{j+1}(z) + \alpha(x) + \sum_{j=1}^{k-1}(-1)^j\alpha^{j+1}(z)\\
        &= x + \sum_{j=1}^k (-1)^j\alpha^j(z) - (-1)^k\alpha^{k+1}(z) = x + \sum_{j=1}^{k+1}(-1)^j\alpha^j(z).
\end{align*}
\end{proof}

\begin{proposition}\label{F affine}
Let $X=\gal{A,\alpha}$ be a finite connected affine quandle. Then for every $x$, $y\in X$ we have
\begin{itemize}
\item[(i)] $|O_f(x,y)| = \min \setof{n\in\mathbb{N}}{\sum_{j=1}^n (-1)^j\alpha^j (x-y/0)=0}$,
\item[(ii)] $(-1)^{|O_f(x,y)|}\alpha^{|O_f(x,y)|}(x-y/0)=x-y/0$,
\item[(iii)] if $n = |O_f(x,y)|$ then $|O_{L_0}(x-y/0)|$ divides $(2n)/\gcd(2,n)$,
\item[(iv)] if $2(x-y/0)=0$ then $|O_f(x,y)|=|O_{L_0}(x-y/0)|$.
\end{itemize}
\end{proposition}
\begin{proof}
(i) By Proposition \ref{Prop:orbit sizes of f}, $|O_f(x,y)|$ is the smallest positive $k$ such that $f_k(x,y)=x$, or, equivalently, $f_{k-1}(x,y)=y/0$. By Lemma \ref{Lm:iterated product}, $f_k(x,y)=x$ if and only if $\sum_{j=1}^k (-1)^j\alpha^j (x-y/0)=0$.

(ii) Let $z=y/0$ and $n=|O_f(x,y)|$. By Lemma \ref{Lm:iterated product} and the above remarks, we have
\begin{displaymath}
    x = f_n(x,y) =x+\sum_{j=1}^n(-1)^j\alpha^j(x-z)\quad\text{and}\quad
    z = f_{n-1}(x,y) =x+\sum_{j=1}^{n-1}(-1)^j\alpha^j(x-z).
\end{displaymath}
Taking the difference of these two equations yields $(-1)^{n}\alpha^n(x-z)=x-z$.

(iii) Since $L_0=\alpha$, part (ii) shows that $|O_{L_0}(x-z)|$ divides $2n$. If $n$ is even, (ii) in fact shows that $|O_{L_0}(x-z)|$ divides $n$.

(iv) Suppose that $2(x-z)=0$. Then $2\alpha^j (x-z)=0$ for every $j$. Part (i) then yields $0 = \sum_{j=1}^n(-1)^j\alpha^j (x-z)=\sum_{j=1}^n \alpha^j (x-z) = \sum_{j=0}^{n-1}\alpha^j(x-z)$, and Proposition \ref{iterative multiplication} implies $\alpha^n(x-z)=x-z$. As $n$ is minimal with $0 = \sum_{j=0}^{n-1}\alpha^j (x-z)$, we deduce $|O_{L_0}(x-z)|=n$.
\end{proof}

We will now express $|O_f(x,y)|$ as a function of $|O_{L_0}(x-y/0)|$ and the order of a certain element of $A$. We present only the case when $|O_{L_0}(x-y/0)|$ is even (since that is all we need later), but the argument in the odd case is similar.

\begin{lemma}\label{Lem: length F even case}
Let $X=\gal{A,\alpha}$ be a finite connected affine quandle and let $x$, $y\in X$. Suppose that $\ell=|O_{L_0}(x-y/0)|$ is even. Then
\begin{displaymath}
    |O_f(x,y)| =\left\{
    \begin{array}{ll}
        k\ell,&\text{ if }|O_f(x,y)|\text{ is even},\\
        k'\ell/2,&\text{ otherwise},
    \end{array}
    \right.
\end{displaymath}
where $k=|\sum_{j=0}^{\ell-1}(-1)^j\alpha^j(x-y/0)|$ and $k' = |\sum_{j=0}^{\ell/2-1}(-1)^j\alpha^j(x-y/0)|$.
\end{lemma}

\begin{proof}
Let $z=x-y/0$, $\ell=|O_{L_0}(z)|$ and $n=|O_f(x,y)|$. Suppose that $n$ is even. Then, by Proposition \ref{F affine}(iii), $n=r\ell$ for some $r$. By Proposition \ref{F affine}(i), we have
\begin{displaymath}
    0 = \sum_{j=1}^{r\ell} (-1)^j\alpha^j(z) = \sum_{j=0}^{r\ell-1} (-1)^j\alpha^j(z)=r \sum_{j=0}^{\ell-1}(-1)^j\alpha^j(z),
\end{displaymath}
where we have used $r\ell$ even in the second equality and $(-1)^\ell \alpha^\ell(z)=z$ (due to $\ell$ even) in the third equality. Moreover, $|\sum_{j=0}^{\ell-1}(-1)^j\alpha^j(z)|=r$ because $r\ell$ is the smallest positive integer for which $\sum_{j=1}^{r\ell} (-1)^j\alpha^j(z)=0$.

Suppose now that $n$ is odd. Then, by Proposition \ref{F affine}(iii), $n=(2s+1)\ell/2$ for some $s$. Since $n$ is odd, we have $\ell/2$ odd and Proposition \ref{F affine}(iii) yields $-z = \alpha^n(z)= \alpha^{(2s+1)\ell/2}(z) = \alpha^{\ell/2}(z)$ and therefore $(-1)^{\ell/2}\alpha^{\ell/2}(z)=z$. Using these observations, Proposition \ref{F affine}(i) implies
\begin{displaymath}
    0 = \sum_{j=1}^{(2s+1)\ell/2}(-1)^j\alpha^j(z) = \sum_{j=0}^{(2s+1)\ell/2-1} (-1)^j\alpha^j(z) =(2s+1)\sum_{j=0}^{\ell/2-1}(-1)^j\alpha^j(z).
\end{displaymath}
We again have $|\sum_{j=0}^{\ell/2-1}(-1)^j\alpha^j(z)|=2s+1$.
\end{proof}

We conclude this section by explicitly calculating $h$ and $|O_h(x,y)|$ in a connected affine quandle.

\begin{lemma}\label{geometric prog}
Let $X=\gal{A,\alpha}$ be a connected affine quandle and let $\beta$ be a $0$-normalized cocycle. Then
\begin{eqnarray}
    h(x,y)&=&(y+x,y),  \label{form:traslation by n}\\
    \beta(ny+x,y)&=&\beta(x,y), \label{omega translations}\\
    \beta(nx,x)&=&1\label{beta x_n x}
\end{eqnarray}
for every integer $n$ and every $x$, $y\in X$. In particular, $|O_h(x,y)|=|y|$.
\end{lemma}

\begin{proof}
Let $x$, $y\in X$, set $z=x\ldiv 0$ and note that $z=x-\alpha^{-1}(x)$. Then
\begin{align*}
    h(x,y)&=(y/(x\ldiv 0)\cdot x,y) =(y/z\cdot x,y)\\
     &=((z+ (1-\alpha)^{-1}(y-z))\cdot x,y)=((1-\alpha)(z)+y-z+\alpha(x),y)\\
     &=(-\alpha(z)+y+\alpha(x),y)=(-\alpha(x)+x+y+\alpha(x),y) \\
     &= (y+x,y).
\end{align*}
The $h$-invariance of $\beta$ (cf. Proposition \ref{Prop:invariance_under_h}) then yields \eqref{omega translations}. With $y=x$, \eqref{omega translations} yields $\beta((n+1)x,x) = \beta(x,x)=1$, which is \eqref{beta x_n x}. Finally, $|O_h(x,y)|=|y|$ is an immediate consequence of \eqref{form:traslation by n}.
\end{proof}

\section{Two classes of simply connected quandles}\label{Sc:Classes}

In this section we show that every finite connected affine quandle over a cyclic group is simply connected (extending a result of Gra\~{n}a for connected quandles of prime order), and that every finite doubly transitive quandle of order different from $4$ is simply connected.

\subsection{Connected affine quandles over cyclic groups}

Gra\~{n}a showed:

\begin{proposition}[{{\cite[Lemma 5.1]{G}}}]\label{Prop:cohomology Z_p}
Let $q$ be a prime and $X=\gal{\mathbb{Z}_q,\alpha}$ an affine quandle. Then $X$ is simply connected.
\end{proposition}

\begin{proof}
Let $x$, $y\in X$ be nonzero elements and let $\beta$ be a $0$-normalized cocycle. Then $y=nx$ for a suitable $n$ and \eqref{beta x_n x} yields $\beta(y,x) = \beta(nx,x)=1$.
\end{proof}

It is known that every connected quandle of prime order $q$ is isomorphic to an affine quandle of the form $\gal{\mathbb{Z}_q,\alpha}$, cf. \cite{EGS}. Proposition \ref{Prop:cohomology Z_p} therefore states that every connected quandle of prime order is simply connected.

Every automorphism of $\mathbb Z_m$ is of the form $\lambda_n$ for some $n$ with $\gcd(m,n)=1$, where $\lambda_n(x) = nx$. Suppose that $\gcd(m,n)=1$. As an immediate consequence of Proposition \ref{connected iff latin}, we see that the affine quandle $\gal{\mathbb Z_m,\lambda_n}$ is connected if and only if $\gcd(m,1-n)=1$. Note that the conditions $\gcd(m,n)=1=\gcd(m,n-1)$ imply that $m$ is odd.

Let $\mathcal U(\mathbb Z_m)$ denote the group of units in the ring of integers modulo $m$.

\begin{proposition}\label{Prop:beta and generators}
Let $X=\gal{\mathbb Z_m,\lambda_n} $ be a connected affine quandle and let $\beta $ be a $0$-normalized cocycle. Then $\beta$ is $u$-normalized for every $u\in\mathcal U(\mathbb Z_m)$, that is, $\beta(x,u)=1$ for every $x\in X$ and $u\in \mathcal U(\mathbb Z_m)$. In addition, $\beta(u,x)=1$ and $\beta(u\cdot x,u\cdot y) = \beta(x,y)$ for every $x$, $y\in X$ and $u\in\mathcal U(\mathbb Z_m)$.
\end{proposition}

\begin{proof}
Let $u\in\mathcal U(\mathbb Z_m)$ and $x\in X$. Then $x = nu$ for a suitable $n$ and we have $\beta(x,u) = \beta(nu,u) = 1$ by \eqref{beta x_n x}, showing that $\beta$ is $u$-normalized. By Proposition \ref{Prop:invariance_under_g}, $\beta(u\cdot x,u\cdot y) = \beta(x,y)$ for every $x$, $y\in X$. Then by Lemma \ref{Lm:LatinEq}, $\beta(u,x)=1$ for every $x\in X$.
\end{proof}

\begin{lemma}\label{Prop: x=u+v}
Let $X=\gal{\mathbb Z_m,\alpha}$ be a connected affine quandle with $m$ odd. Then for every $x\in X$ there are $u$, $v\in \mathcal U(\mathbb Z_m)$ such that $x=u\cdot v$.
\end{lemma}
\begin{proof}
Since $u+v=(1-\alpha)^{-1}(u)\cdot \alpha^{-1}(v)$ and $\alpha$ and $1-\alpha$ are automorphisms of $(\mathbb Z_m,+)$, it suffices to prove that for every $x\in\mathbb Z_m$ there are $u$, $v\in\mathcal U(\mathbb Z_m)$ such that $x=u+v$. This is well-known and can be established as follows:

Let $m=p_1^{n_1}\cdots p_r^{n_r}$, where $p_1$, $\dots$, $p_r$ are distinct primes. By the Chinese remainder theorem, $\mathbb Z_m\cong \mathbb Z_{p_1^{n_1}}\times\cdots\times\mathbb Z_{p_r^{n_r}}$ and also $\mathcal U(\mathbb Z_m)\cong \mathcal U(\mathbb Z_{p_1^{n_1}})\times\cdots\times\mathcal U(\mathbb Z_{p_r^{n_r}})$. It therefore suffices to prove the claim when $m=p^n$ for a prime $p$. Consider $x=ap+b$, where $0\le b<p$. If $x$ is invertible (that is, $b\ne 0$), then so is $2x$ (since $m$ is odd), and $x = 2x+(-x)$ does the job. If $x$ is not invertible, then $x = ap = (ap+1)+(-1)$ finishes the proof.
\end{proof}

\begin{theorem}\label{Teo: H(Z_m)=1}
Let $X=\gal{\mathbb Z_m,\lambda_n}$ be a connected affine quandle. Then $X$ is simply connected.
 \end{theorem}

\begin{proof}
Let $x$, $y\in X$ and let $\beta$ be a $0$-normalized cocycle. By Proposition \ref{Prop: x=u+v}, we can write $x=u\cdot v$ for some invertible elements $u$, $v$. By Proposition \ref{Prop:beta and generators}, $\beta(x,y) = \beta(u \cdot v,u\cdot u\backslash y)=\beta (v,u\backslash y)=1$.
\end{proof}

\subsection{Doubly transitive quandles}\label{Ss:DoublyTransitive}

Finite doubly transitive quandles can be characterized as follows:

\begin{theorem}[{{\cite[Corollary 4]{V}}}]\label{Th:DoublyTransitiveQuandle}
Let $X$ be a finite quandle. Then $\lmlt{X}$ is doubly transitive if and only if $X\cong\gal{\mathbb Z_q^n,\alpha}$ for some prime $q$, some $n>0$ and $\alpha\in\aut{\mathbb Z_q^n}$ with $|\alpha| =|X|-1$.
\end{theorem}

\begin{lemma}\label{Lm:DoublyTransitiveIdempotentGroupoid}
Let $X$ be a finite idempotent groupoid with a doubly transitive automorphism group. Then $X$ is semiregular and the parameter $s$ (length of any nontrivial orbit of any left translation) is a divisor of $|X|-1$.
\end{lemma}
\begin{proof}
For any $x\in X$ we have $L_x(x)=x$ by idempotence. Given $x\ne y$ with $L_x^k(y)=y$ and some $v\ne w$, let $\varphi\in\aut{X}$ be such that $\varphi(x)=v$, $\varphi(y)=w$. Then $w = \varphi(y) = \varphi L_x^k(y) = L_{\varphi(x)}^k(\varphi(y)) = L_v^k(w)$. Hence $X$ is semiregular, and the rest follows.
\end{proof}

Combining Theorem \ref{Th:DoublyTransitiveQuandle} and Lemma \ref{Lm:DoublyTransitiveIdempotentGroupoid}, we see that if $X\cong\gal{\mathbb Z_q^n,\alpha}$ is a finite doubly transitive quandle, then $\alpha$ is an $(|X|-1)$-cycle (since $\alpha$ cannot be trivial by Proposition \ref{connected iff latin}).

\begin{proposition}\label{same F}
Let $X$ be a finite doubly transitive quandle. Then there is an integer $F>1$ such that
\begin{displaymath}
    |O_f(x,y)| = \left\{
    \begin{array}{rr}
        1,&\text{ if $x=y/0$},\\
        F,&\text{ otherwise}.
    \end{array}\right.
\end{displaymath}
\end{proposition}
\begin{proof}
By Proposition \ref{Prop:orbit sizes of f}, $|O_f(y/0,y)|=1$. Suppose that $x$, $y$, $v$, $w\in X$ are such that $x\ne y/0$ and $v\ne w/0$. Let $\varphi\in\lmlt{X}\le\aut{X}$ be such that $\varphi(x)=v$ and $\varphi(y/0)=w/0$. Suppose that $|O_f(x,y)|=k$ is even, the odd case being similar. By Lemma \ref{f orbits} and Proposition \ref{Prop:orbit sizes of f}, we have $(L_x L_{y/u})^{k/2}(x)=x$. Then $v=\varphi(x) = \varphi(L_xL_{y/0})^{k/2}(x) = (L_vL_{w/0})^{k/2}(v)$, and thus $|O_f(v,w)|\le k$.
\end{proof}

\begin{corollary}\label{F for 2trans}
Let $X=\gal{\mathbb Z_q^n,\alpha}$ be a doubly transitive quandle. Then
\begin{enumerate}
\item[(i)] $\mathcal O_g^0$ and $\mathcal O_g^f$ are singletons,
\item[(ii)] $(\bigcup\mathcal O_g^0)\cap(\bigcup\mathcal O_g^f)=\emptyset$ if $|X|>3$,
\item[(iii)] if $q=2$ then $F=|X|-1$,
\item[(iv)] if $q>2$ then $F=|X|-1$ if it is even and $F=(|X|-1)/2$ if it is odd.
\end{enumerate}
\end{corollary}
\begin{proof}
(i) Since $\alpha=L_0$ is an $(|X|-1)$-cycle, we have $|O_g(x,y)|=|X|-1$ whenever $(x,y)\ne(0,0)$ by Proposition \ref{Prop:orbit sizes of g}. The claim follows by Lemma \ref{Lm:some orbits of g}.

(ii) By Lemma \ref{Lm:some orbits of g}, $(x,y)\in (\bigcup\mathcal O_g^0)\cap(\bigcup\mathcal O_g^f)$ if and only if $x\cdot (x\cdot 0)=0$ and $y=x\cdot 0$. We have $x\cdot (x\cdot 0) = (1+\alpha)(1-\alpha)(x)$. Since $1-\alpha\in\aut{A}$, we have $x\cdot (x\cdot 0)=0$ if and only if $\alpha(x)=-x$. Then $\alpha^2(x)=x$, so $|X|-1=|\alpha|\le 2$.

(iii) According to Proposition \ref{same F}, it suffices to compute the length of an $f$-orbit for some $(x,y)$ with $y\ne x\cdot 0$. We have $2x=0$ for every $x\in X$, and hence $F=|X|-1$ by Proposition \ref{F affine}(iv).

(iv) Since $q>2$, $|O_{L_0}(x)|=|X|-1$ is even for every $x\in X$. By Proposition \ref{Prop:orbit sizes of f}, $F\le |X|-1$. If $F$ is even then $F=k(|X|-1)$ by Lemma \ref{Lem: length F even case} and thus $F=|X|-1$. If $F$ is odd then the same lemma yields $F = k(|X|-1)/2$, the case $F=|X|-1$ cannot occur since $|X|-1$ is even, and thus $F=(|X|-1)/2$.
\end{proof}

\begin{lemma}\label{Lm:Technical}
Let $X=\gal{\mathbb Z_q^n,\alpha}$ be a doubly transitive quandle and $\beta$ a $0$-normalized cocycle.
\begin{enumerate}
\item[(i)] If $F=|X|-1$ then $\beta(x,y)=1$ for every $(x,y)\not\in\bigcup (\mathcal O_g^f\cup\mathcal O_g^0)$.
\item[(ii)] If $\beta(x,y)=1$ for every $(x,y)\not\in\bigcup (\mathcal O_g^f\cup\mathcal O_g^0)$ and, in addition, there is $0\ne z\in X$ such that $\beta(z,z\ldiv 0)=1$, then $\beta=1$.
\end{enumerate}
\end{lemma}
\begin{proof}
(i) We have $\beta(x,x)=1$ for any $x\in X$. Suppose that $(0,0)\ne (x,y)\in (X\times X)\setminus \bigcup(\mathcal O_g^f\cup\mathcal O_g^0)$. Then $(x,y)$ is not a fixed point of $f$, and neither is $(xy,xy)$, since $xy\ne 0$. The fiber $p^{-1}(xy)$ contains both $(x,y)$ and $(xy,xy)$, and it is a union of $f$-orbits. By assumption, one of the orbits has size $|F|-1$, forcing the remaining element of $p^{-1}(xy)$ to be a fixed point of $f$. Hence $(x,y)\in O_f(xy,xy)$. Then $\beta(xy,xy)=1$ implies $\beta(x,y)=1$.

(ii) Suppose that $\beta(x,y)=1$ for every $(x,y)\not\in\bigcup (\mathcal O_g^f\cup\mathcal O_g^0)$, and $\beta(z,z\ldiv 0)=1$ for some $0\ne z\in X$. Then $\beta=1$ on $\bigcup\mathcal O_g^0$ since $\mathcal O_g^0$ is a singleton by Corollary \ref{F for 2trans}. Finally, let $(x,y)\in\bigcup\mathcal O_g^f$. By Corollary \ref{action of omega on diag}, $h(x,y)\not\in \bigcup\mathcal O_g^f$ and thus $\beta(x,y)=\beta(h(x,y))=1$.
\end{proof}

\begin{lemma}\label{Lemma: affine order 2}
Let $X=\gal{ A,\alpha }$ be a latin affine quandle and $0\ne x\in X$. Then
\begin{enumerate}
\item[(i)] $h(x,x\ldiv 0)\in\bigcup\mathcal O_g^f$ if and only if $2x=0$,
\item[(ii)] $O_g( 0/ ( 0/x ) ,x ) \in\mathcal  O_g^{f} $ if and only if $ ( \alpha ^{2}+\alpha -1) ( x) =0$.
\end{enumerate}
\end{lemma}

\begin{proof}
(i) Recall that $(x,y)\in\bigcup\mathcal O_g^f$ if and only if $y=x\cdot 0$. We have $h(x,x\ldiv 0)=(x+x\ldiv 0,x\ldiv 0)$ and $x\ldiv 0=(1-\alpha^{-1})(x)$. Then $h(x,x\ldiv 0)\in \mathcal O_g^f$ if and only if $(x+x\ldiv 0)\cdot 0 = x\ldiv 0$, which is equivalent to $(1-\alpha)((1-\alpha^{-1})(x)+x) = (1-\alpha^{-1})(x)$. This is easily seen to be equivalent to $(1-\alpha)(2x)=0$. Since $1-\alpha\in\aut{A}$, the last condition is equivalent to $2x=0$.

(ii) Note that $(0/(0/x))\cdot 0 =x$ is equivalent to $x/0\cdot 0/x  = 0$. Also note that $x/0=(1-\alpha)^{-1}(x)$ and $0/x=-\alpha( 1-\alpha  ) ^{-1} (x)$. Then $x/0\cdot 0/x  = 0$ holds if and only if $x-\alpha^2(1-\alpha)^{-1}(x) = 0$, which is equivalent to $(\alpha^2+\alpha-1)(x)=0$.
\end{proof}

\begin{theorem}\label{odd_2trans}
Let $X=\gal{\mathbb{Z}_{q}^{n},\alpha} $ be a doubly transitive quandle with $q\ge 3$. Then $X$ is simply connected.
\end{theorem}

\begin{proof}
If $n=1$ then $X$ is simply connected by Theorem \ref{Teo: H(Z_m)=1}. Suppose that $n>1$ and let $\beta$ be a $0$-normalized cocycle. Recall that $X$ is semiregular with $\alpha$ an $(|X|-1)$-cycle. Hence $|O_g(x,y)|=|X|-1$ for every $(0,0)\ne(x,y)\in X\times X$. By Lemma \ref{F on Delta} and Corollary \ref{F for 2trans}, $|O_f(O_g(x,y))| = |O_f(x,y)|\in\{1$, $F\}$ and $F\in\{|X|-1$, $(|X|-1)/2\}$. By Lemmas \ref{F on Delta} and \ref{geometric prog}, $|O_h(O_g(x,y))| = |O_h(x,y)|=|y|\in\{1,q\}$ for any $x$, $y\in X$. Moreover, by Corollary \ref{action of omega on diag}, $\mathcal O_g^0$ and $\mathcal O_g^f$ are disjoint singletons.

Suppose that $F=|X|-1$. By Lemma \ref{Lm:Technical}(i), $\beta=1$ on the complement of $O=\bigcup(\mathcal O_g^0\cup\mathcal O_g^f)$. Let $0\ne z\in X$. By Lemma \ref{Lm:Technical}(ii), we will be done if we show that $\beta(z,z\ldiv 0)=1$. By Corollary \ref{action of omega on diag}, $h(z,z\ldiv 0)\notin\bigcup\mathcal O_g^0$. Since $q\ge 3$, Lemma \ref{Lemma: affine order 2}(i) yields $h(z,z\ldiv 0)\not\in\bigcup\mathcal O_g^f$. Hence $h(z,z\ldiv 0)\in O$ and $\beta(z,z\ldiv 0) = \beta(h(z,z\ldiv 0)) = 1$.

For the rest of the proof suppose that $F=(|X|-1)/2$. The sets $O_g(u,u)$, $\mathcal O_g^f$ and $\mathcal O_g^0$ account for $1 + 2(|X|-1) = 2|X|-1$ elements of $X\times X$ and for all fixed points of $f$, leaving $|X|^2-(2|X|-1)=(|X|-1)^2$ points unaccounted for. The unaccounted points thus form two $\langle f,g\rangle$-orbits, each of size $F(|X|-1)$, say $\mathcal O_g^1$ and $\mathcal O_g^2$. We can certainly take $\mathcal O_g^1 = O_f(O_g(0,x))$ for some $0\ne x$. Since $\beta(0,x)=1$ by Lemma \ref{Lm:LatinEq}, we have $\beta=1$ on $\mathcal O_g^1$. If we can show that $\beta=1$ on $\mathcal O_g^2$, too, then we can finish as in the case $F=|X|-1$, completing the proof.

We will now show that if $q\ne 3$ then the induced action of $h$ on $\mathcal O_g$ has only two orbits, namely $\{O_g(0,0)\}$ and its complement. This will finish the proof for $q\ne 3$. Since $h$ has no fixed points in the set $\mathcal O_g^0\cup\mathcal O_g^f$ of size $2$, it suffices to consider the following situations:

(a) Suppose that $h$ acts on $\mathcal O_g^f\cup\mathcal O_g^1$ and thus also on $\mathcal O_g^0\cup\mathcal O_g^2$, both sets of size $F+1$. Let $0\ne x\in X$. Since $O_g(x,0)$ is fixed by $h$ and belongs to one of these sets, we see that $h$ acts on a set of size $F$, hence $q$ divides $F=(|X|-1)/2$, hence $q$ (being odd) divides $|X|-1$, a contradiction. We reach a similar contradiction if $h$ acts on $\mathcal O_g^f\cup\mathcal O_g^2$ and $\mathcal O_g^0\cup  \mathcal O_g^1$.

(b) Suppose that $h$ acts on $\mathcal O_g^0\cup \mathcal O_g^f\cup \mathcal O_g^1$ and thus also on $\mathcal O_g^2$, sets of size $F+2$ and $F$, respectively. Since $\mathcal O_g^2$ contains no fixed-points of $h$, we reach a contradiction as in (a).

(c) Suppose that $h$ acts on $\mathcal O_g^0\cup \mathcal O_g^f\cup \mathcal O_g^2$ and thus also on $\mathcal O_g^1$, sets of size $F+2$ and $F$, respectively. Once we account for $O_g(x,0)\in\mathcal O_g^1$, we see that $h$ acts on a set of size $F+2=(|X|+3)/2$ and on a set of size $F-1=(|X|-3)/2$, forcing $q=3$.

For the rest of the proof we can therefore assume that $q=3$. Let $0\ne x\in X$. Setting $x=z$ and $y=0$ in \eqref{CCC} yields $\beta(x\cdot 0,x)=\beta(x,0\cdot x)$. We also have $(x\cdot 0,x)$, $(x,0\cdot x) \not\in O$. Suppose for a while that $(x\cdot 0,x)$ and $(x,0\cdot x)$ are in the same $f$-orbit, that is,
\begin{displaymath}
    (x\cdot 0,x) = f^k(x,0\cdot x) = (f_k(x,0\cdot x),f_{k-1}(x,0\cdot x)\cdot 0)
\end{displaymath}
for some $k\ge 1$. Note that $(0\cdot x)/0 = (1-\alpha)^{-1}\alpha(x)$. Comparing coordinates, Lemma \ref{Lm:iterated product} yields
\begin{gather*}
    (1-\alpha)(x) = x\cdot 0 = f_k(x,0\cdot x) = x + \sum_{j=1}^k (-1)^j\alpha^j(x-(1-\alpha)^{-1}\alpha(x)),\\
    x = f_{k-1}(x,0\cdot x)\cdot 0 = (1-\alpha)(x+\sum_{j=1}^{k-1}(-1)^j\alpha^j(x-(1-\alpha)^{-1}\alpha(x))).
\end{gather*}
Applying $1-\alpha$ to the first identity and using $q=3$ then yields
\begin{align*}
    (1-\alpha)^2(x) &= (1-\alpha)(x) + \sum_{j=1}^k (-1)^j\alpha^j((1-\alpha)(x)-\alpha(x))\\
        &= (1-\alpha)(x) +\sum_{j=1}^k (-1)^j\alpha^j(1+\alpha)(x),
\end{align*}
while the second identity can be rewritten as
\begin{displaymath}
    x = (1-\alpha)(x) + \sum_{j=1}^{k-1}(-1)^j\alpha^j((1-\alpha)(x)-\alpha(x)) = (1-\alpha)(x) +\sum_{j=1}^{k-1}(-1)^j\alpha^j(1+\alpha)(x).
\end{displaymath}
Subtracting the two last identities now gives $(1-\alpha)^2(x) - x = (-1)^k\alpha^k(1+\alpha)(x)$. Since $(1-\alpha)^2 = 1-2\alpha+\alpha^2 = 1+\alpha+\alpha^2$, we can rewrite this as $\alpha(1+\alpha)(x) = (-1)^k\alpha^k(1+\alpha)(x)$. Noting that $1+\alpha\in\aut{A}$ (if $\alpha(x)=-x$ then $\alpha^2=1$ and $|X|=3$) and, canceling, we finally get $x = (-1)^k\alpha^{k-1}(x)$. If $k$ is even, we deduce $k\equiv 1\pmod{|X|-1}$, thus also $k\equiv 1\pmod F$, but then $(x\cdot 0,x) = f^k(x,x\cdot 0)=(x,x\cdot 0)$ implies $x\cdot 0 = x$, $x=0$, a contradiction. If $k$ is odd, we deduce $2(k-1)\equiv 0\pmod { |X|-1}$, therefore $k\equiv 1\pmod F$, and we reach the same contradiction.

Consequently, the elements $(x,x\cdot 0)$ and $(x\cdot 0,x)$ are not in the same $f$-orbit, hence one of them lies in $\bigcup \mathcal O_g^1$ while the other in $\bigcup \mathcal O_g^2$, and we see that $\beta=1$ on $\mathcal O_g^2$.
\end{proof}

\begin{theorem}\label{even_2trans}
Let $X=\gal{ \mathbb{Z}_{2}^{n},\alpha } $ be a
doubly transitive quandle with $n\neq 2$. Then $X$ is simply connected.
\end{theorem}

\begin{proof}
By Corollary \ref{F for 2trans}(i), $F=|X|-1$. Then by Lemma \ref{Lm:Technical}, $\beta(x,y)=1$ for every $O_g(x,y)\not\in  \mathcal O_g^0\cup \mathcal O_g^f$ and it suffices to show that $\beta(z,z\ldiv 0)=1$ for some $z\ne 0$. Note that any element $(z,z\ldiv 0)$ can be written as $(0/y,y)$ by setting $y=z\ldiv 0$. By Proposition \ref{some prop of beta}, $\beta(0/y,y)=\beta(0/(0/y),y)$ and $O_g(0/(0/y),y) \notin\mathcal O_g^0$. If we show that $O_g(0/(0/y),y)\not\in\mathcal O_g^f $ for some $y\ne 0$, then $\beta(0/y,y)=\beta(0/(0/y),y)=1$ and we are through.

By Lemma \ref{Lemma: affine order 2}(ii), it suffices to show that $(\alpha^2+\alpha+1)(y)=(\alpha^2+\alpha-1)(y)\ne 0$, which is equivalent to $\alpha^3(y)\ne y$, and this follows from the fact that $\alpha$ is an $(|X|-1)$-cycle (here we use $|X|\ne 4$).
\end{proof}

We have now proved Theorem \ref{Th:Main}. We will show in Section \ref{Sc:Clauwens} that a doubly transitive quandle of order $4$ is not simply connected.

\section{A short proof of Theorem \ref{Th:Main}}\label{Sc:Clauwens}

In this section we prove Theorem \ref{Th:Main} once more, this time using results of Clauwens \cite{Clauwens} on the fundamental group of affine quandles.

Clauwens showed how to explicitly calculate the fundamental groups of affine quandles. Following \cite[Definition 1]{Clauwens}, let $G=(G,+)$ be an abelian group and $\mathcal Q(G,\alpha)$ an affine quandle. Let
\begin{align*}
    I(G,\alpha) &= \genof{x\otimes y-y\otimes \alpha(x)}{x,\,y\in G},\\
    S(G,\alpha) &= (G\otimes G)/I(G,\alpha),\\
    F(G,\alpha) &= \mathbb Z\times G\times S(G,\alpha),
\end{align*}
where the operation on $F(G,\alpha)$ is given by
\begin{displaymath}
    (k,x,a)(m,y,b) = (k+m,\,\alpha^m(x)+y,\,a+b+(\alpha^m(x)\otimes y+I(G,\alpha))).
\end{displaymath}
Then we have:

\begin{theorem}[{{\cite[Theorem 1 and page 4]{Clauwens}}}]\label{Th:Clauwens}
Let $G=(G,+,0)$ be an abelian group and $\mathcal Q(G,\alpha)$ an affine quandle. Then the groups $\mathrm{Adj}(\mathcal Q(G,\alpha))$ and $F(G,\alpha)$ are isomorphic, and the groups $\pi_1(\mathcal Q(G,\alpha),0)$ and $S(G,\alpha)$ are isomorphic.
\end{theorem}

We are ready to prove Theorem \ref{Th:Main}. Let $X=\mathcal Q(G,\alpha)$ be a finite connected affine quandle. By Proposition \ref{Pr:SC} and Theorem \ref{Th:Clauwens}, $X$ is simply connected if and only if $S(G,\alpha)$ is trivial.

Suppose first that $G=\mathbb Z_n$ is cyclic. Then $G\otimes G = \mathbb Z_n\otimes\mathbb Z_n\cong \mathbb Z_n$ is generated by $1\otimes 1$. For $y\in G$ we have $1\otimes (1-\alpha)(y) = 1\otimes y - 1\otimes \alpha(y) = y\otimes 1 - 1\otimes \alpha(y)\in I(G,\alpha)$. Since $X$ is connected, the homomorphism $1-\alpha$ is bijective by Proposition \ref{connected iff latin}. In particular, $1\otimes 1\in I(G,\alpha)$ and $S(G,\alpha)$ is trivial.

Now suppose that $X=\mathcal Q(G,\alpha)$ is doubly transitive, $G$ is not cyclic, and $|X|\ne 4$. By the argument given in Subsection \ref{Ss:DoublyTransitive}, $G=\mathbb Z_p^n$ for some prime $p$ and $n>1$, and $\alpha$ is a cycle of length $|X|-1$.

Let us write $u\equiv v$ if $u$, $v\in G\otimes G$ coincide modulo $I(G,\alpha)$. For every $x\in G$ we have
\begin{align*}
    0&\equiv x\otimes (x+\alpha(x)) -(x+\alpha(x))\otimes \alpha(x) \\
    &= x\otimes x +x\otimes \alpha(x) - x\otimes \alpha(x) - \alpha(x)\otimes\alpha(x)
    = x\otimes x - \alpha(x)\otimes\alpha(x).
\end{align*}
Therefore $x\otimes x \equiv \alpha(x)\otimes \alpha(x)$ for every $x$. Since $\alpha$ is a cycle of length $|X|-1$, we conclude that there is $e\in G$ such that
\begin{equation}\label{Eq:xx}
    x\otimes x\equiv e\quad\text{for every $0\ne x\in G$}.
\end{equation}
If $x$, $y\in G$ are such that $x\ne 0\ne y$ and $x\ne y$, equation \eqref{Eq:xx} implies
\begin{displaymath}
    e\equiv (x-y)\otimes (x-y) = x\otimes x - x\otimes y -y\otimes x + y\otimes y
    \equiv 2e - x\otimes y - y\otimes x.
\end{displaymath}
Therefore
\begin{equation}\label{Eq:Plus}
    x\otimes y + y\otimes x \equiv e\quad\text{for every $x\ne 0\ne y$ with $x\ne y$}.
\end{equation}
We proceed to show that $e\equiv 0$.

Suppose that $p\ne 2$. Since $|X|>4$, there are $x$, $y\in G$ such that $x\ne 0\ne y$ and $x\ne \pm y$. Then \eqref{Eq:xx} implies
\begin{displaymath}
    e\equiv (x+y)\otimes (x+y) = x\otimes x + x\otimes y + y\otimes x + y\otimes y
    \equiv 2e + x\otimes y + y\otimes x
\end{displaymath}
and we deduce $x\otimes y + y\otimes x \equiv -e$. But \eqref{Eq:Plus} holds for our choice of $x$, $y$ as well, and thus $e\equiv 0$.

Suppose now that $p=2$. Since $|X|>4$, there are distinct and nonzero $x$, $y$, $z\in G$ such that $x+y+z\ne 0$. Then \eqref{Eq:xx} and \eqref{Eq:Plus} imply
\begin{align*}
    e&\equiv (x+y+z)\otimes(x+y+z) \\
    &= x\otimes x + y\otimes y + z\otimes z  + (x\otimes y + y\otimes x) + (x\otimes z + z\otimes x) + (y\otimes z + z\otimes y)\\
    &\equiv 6e\equiv 0,
\end{align*}
and we again conclude that $e\equiv 0$.

Let us continue the proof with $p$ arbitrary. We have shown that $x\otimes x\equiv 0$ for every $x$. The calculations leading to \eqref{Eq:Plus} can now be repeated for any $x$, $y$, and we obtain $x\otimes y \equiv - y\otimes x$ for every $x$, $y$. Hence
\begin{equation}\label{Eq:alpha}
    0\equiv x\otimes y -y\otimes \alpha(x)\equiv x\otimes y + \alpha(x)\otimes y
    = (x+\alpha(x))\otimes y
\end{equation}
for every $x$, $y\in G$. We claim that $1+\alpha$ is bijective. Indeed, suppose that $(1+\alpha)(x)=0$ for some $x\ne 0$. Then $\alpha(x)=-x$ and $\alpha^2(x)=x$, a contradiction with $\alpha$ being a cycle of length $|X|-1$ (using $|X|>3$). Now \eqref{Eq:alpha} shows that $x\otimes y\equiv 0$ for every $x$, $y\in G$, and Theorem \ref{Th:Main} is proved.

\begin{example}\label{Ex:Clauwens4}
This example shows that a doubly transitive quandle $\mathcal Q(G,\alpha)$ of order $4$ is not simply connected. We will calculate $I(G,\alpha)$. Let $\{e_1,e_2\}$ with $e_1=\binom{1}{0}$ and $e_2=\binom{0}{1}$ be a basis of $G=\mathbb Z_2^2$, and suppose without loss of generality that $\alpha=\binom{1\ 1}{1\ 0}$. Then $\alpha(e_1) = e_1+e_2$, $\alpha(e_2) = e_1$ and $\alpha(e_1+e_2)=e_2$. Calculating in $G\otimes G$, we get
\begin{align*}
    e_1\otimes e_1 + e_1\otimes \alpha(e_1)&= e_1\otimes e_2,\\
    e_2\otimes e_1 + e_1\otimes \alpha(e_2)&=  e_1\otimes e_1+e_2\otimes e_1 ,\\
    e_1\otimes e_2 + e_2\otimes \alpha(e_1) &= e_1\otimes e_2 + e_2\otimes e_1 + e_2\otimes e_2,\\
    e_2\otimes e_2 + e_2\otimes \alpha(e_2) &= e_2\otimes e_1 + e_2\otimes e_2.
\end{align*}
Hence $I(G,\alpha)$ is the span of $e_1\otimes e_1+e_2\otimes e_1$, $e_2\otimes e_1+e_2\otimes e_2$ and $e_1\otimes e_2$. Since $G\otimes G\cong\mathbb Z_2^4$, the quandle $\mathcal Q(G,\alpha)$ is not simply connected.
\end{example}

\section{Constant cocycles with coefficients in arbitrary groups}\label{Sc:ArbitraryGroup}

Following \cite{AG}, we have defined constant quandle cocycles as mappings $\beta: X\times X\to \sym{S}$ but the definition makes sense for any group $G$.

\begin{definition}
Let $X$ be a quandle and $G$ a group. Let
\begin{displaymath}
    Z^2_c(X,G)=\setof{\beta:X\times X\to G}{\beta \text{ satisfies } \eqref{CCC} \text{ and } \eqref{CQC}}.
\end{displaymath}
For $\beta$, $\beta'\in Z^2_c(X,G)$, we write $\beta\sim\beta'$ if there exists a mapping $\gamma:X\to G$ such that
\begin{equation*}
    \beta'(x,y) =\gamma(xy) \beta(x,y) \gamma(y)^{-1}
\end{equation*}
holds for every $x$, $y\in X$. Then $H^2_c(X,G) =Z^2_c(\X,G)/\sim$ is the \emph{second constant cohomology set of $X$ with coefficients in $G$}.
\end{definition}

A careful look at all our results shows that all calculations can be carried out over an arbitrary group $G$, not just over symmetric groups.

\begin{proposition}\label{Pr:Embedding}
Let $X$ be a quandle and $G$ a group. Then $Z^2_c(X,G)$ embeds into $Z^2_c(X,\sym{G})$.
\end{proposition}
\begin{proof}
Let $\lambda_G:G\to\sym{G}$, $\lambda_G(g)(h) = gh$ be the left regular representation of $G$. Define
\begin{equation}\label{Eq:Embedding}
    j: Z^2_c(X,G)\to  Z^2_c(X,\sym{G}),\quad j(\beta)(x,y)=\lambda_G(\beta(x,y)).
\end{equation}
Then it is easy to see that $j$ is injective.
\end{proof}

Let us show that the embedding $j$ in \eqref{Eq:Embedding} induces a map $j:H^2_c(X,G)\to H^2_c(X,\sym{G})$. Suppose that $\beta$, $\beta'\in Z^2_c(X,G)$ are cohomologous and let $\gamma:X\to G$ be such that $\gamma(xy)\beta(x,y)\gamma(y)^{-1} = \beta'(x,y)$ for every $x$, $y\in X$. We claim that $j(\beta)$, $j(\beta')\in Z^2_c(X,\sym{G})$ are cohomologous via $j(\gamma)$, defined by $j(\gamma)(x)=\lambda_G(\gamma(x))$. Indeed, for every $a\in G$ and every $x$, $y\in X$ we have
\begin{displaymath}
    j(\gamma)(xy)j(\beta)(x,y)(j(\gamma)(y))^{-1}(a){=}\gamma(xy)\beta(x,y)\gamma(y)^{-1}a{=}\beta'(x,y)a{=}j(\beta')(x,y)(a).
\end{displaymath}

However, the induced map $j:H^2_c(X,G)\to H^2_c(X,\sym{G})$ is not necessarily an embedding as we shall see in Example \ref{Ex:NotEmbedding}. We start by having another look at $4$-element doubly transitive quandles.

\begin{example}\label{Ex:Details}
Let $X= \gal{ \mathbb{Z}_{2}^{2},\alpha}$ be a doubly transitive quandle and $G$ a group. Suppose that $\alpha$, $e_1$, $e_2$ are as in Example \ref{Ex:Clauwens4}. We claim that
\begin{equation}\label{4elements_cohomology}
    H^2_c(X,G) =\setof{[\beta_a]_\sim}{a\in G,\,a^2=1},
\end{equation}
where $\beta_a$ is given by
\begin{equation}\label{beta X=4}
    \begin{array}{c|cccc}
    &0&e_1&e_1+e_2&e_2\\
    \hline
    0& 1 & 1 & 1 & 1 \\
    e_1& 1 & 1 & a & a \\
    e_1+e_2& 1 & a & 1 & a \\
    e_2& 1 & a & a & 1%
    \end{array},
\end{equation}
and, moreover, $\beta_a\sim \beta_b$ if and only if $a$ and $b$ are conjugate in $G$.

To see this, first recall that $x\cdot y = x + \alpha(-x+y) = x+\alpha(x+y)$ and check that $\beta_a$ defined by \eqref{beta X=4} is a $0$-normalized cocycle. Conversely, suppose that $\beta\in Z^2_c(X,G)$ is $0$-normalized. Then $\beta(x,0)=1$ for every $x\in X$. Since $\beta$ is $g$-invariant by Proposition \ref{Prop:invariance_under_g}, we also have $\beta(0,x)=1$ for every $x\in X$ by Lemma \ref{Lm:LatinEq}, and $\beta(x,y) = \beta(g(x,y)) = \beta(0\cdot x,0\cdot y) = \beta(\alpha(x),\alpha(y))$ for every $x$, $y\in X$. By Proposition \ref{Prop:invariance_under_h}, $\beta$ is $h$-invariant, and by Lemma \ref{geometric prog} we have $h(x,y)=(y+x,y)$, so $\beta(x,y)=(y+x,y)$ for every $x$, $y\in X$. Applying $g$, $h$ as indicated, we get
\begin{displaymath}
    \beta(e_1,e_2) \overset{g}{=} \beta(e_1+e_2,e_1) \overset{g}{=} \beta(e_2,e_1+e_2)
    \overset{h}{=} \beta(e_1,e_1+e_2) \overset{g}{=} \beta(e_1+e_2,e_2) \overset{g}{=} \beta(e_2,e_1),
\end{displaymath}
so there is $a\in G$ such that $\beta=\beta_a$. Setting $x=z$ in \eqref{CCC} yields $\beta(xy,x)=\beta(x,yx)\beta(y,x)$ and thus
\begin{align*}
    1 &= \beta(0,e_1) = \beta(e_1\cdot(e_1+e_2),e_1)\\
    &= \beta(e_1,(e_1+e_2)\cdot e_1)\beta(e_1+e_2,e_1) = \beta(e_1,e_2)\beta(e_1+e_2,e_1)=a^2.
\end{align*}
Finally, by Proposition \ref{Prop:cocycle normalized}, $\beta_a\sim\beta_b$ if and only if $a$ and $b$ are conjugate in $G$.
\end{example}

\begin{remark}
In \cite{Mo}, the second cohomology of $X=\gal{\mathbb{Z}_{2}^{2},\alpha}$ was calculated when $G$ is the additive group of a finite field or $G\in\{\mathbb Z,\mathbb Q\}$. Since $G$ is abelian here, $H^2_c(X,G)$ is also an abelian group under the operation $(\beta+\delta)(a,b)=\beta(a,b)+\delta(a,b)$. Our calculations in Example \ref{Ex:Details} agree with those of \cite[Example 2 and Corollary 1.1]{Mo}. We have $H^2_c(X,G)=1$ if $G\in\{\mathbb Z,\mathbb Q\}$ or $G=\mathbb Z_p^k$ with $p$ odd, and $H^2_c(X,\,\mathbb Z_2^k)\cong \mathbb Z_2^k$.
\end{remark}

\begin{example}\label{Ex:NotEmbedding}
Let $X=\gal{ \mathbb{Z}_{2}^{2},\alpha}$ be a doubly transitive quandle and let $G=\mathbb Z_2^2$. By Example \ref{Ex:Details}, every $a\in G$ yields $\beta_a\in Z^2_c(X,G)$, and $\beta_a\sim \beta_b$ if and only if $a=b$ since $G$ is abelian. Example \ref{Ex:Details} also shows that every $\sigma\in \sym{G}$ with $\sigma^2=1$ yields $\beta_\sigma$ in $Z^2_c(X,\sym{G})$, and $\beta_\sigma=\beta_\tau$ if and only if $\sigma$, $\tau$ have the same cycle structure.

Consider now the embedding $j:Z^2_c(X,G)\to Z^2_c(X,\sym{G})$ and note that for every $a\in G$ we have $j(\beta_a) = \beta_{\lambda_G(a)}$. If $a$, $b\in G$ are distinct nonzero elements of $G$ then $\beta_a\not\sim\beta_b$, but $j(\beta_a)\sim j(\beta_b)$ because $\lambda_G(a)$, $\lambda_G(b)$ have the same cycle structure. We conclude that $j$ does not induce an embedding of $H^2_c(X,G)$ into $H^2_c(X,\sym{G})$.
\end{example}

If $X$ is latin, the embedding $j$ commutes with the normalization procedure described in Proposition \ref{Prop:cocycle normalized}.

\begin{proposition}\label{j commutes with normalization}
Let $X$ be a latin quandle, $G$ a group and $\beta\in H^2_c(X,G)$. Then $j(\beta_u)=j(\beta)_u$ for every $u\in X$.
\end{proposition}

\begin{proof}
We have
\begin{eqnarray*}
    j(\beta_u)(x,y)&=&\lambda_G(\beta_u(x,y))=\lambda_G(\beta((xy)/u,u)^{-1}\beta(x,y)\beta(y/u,u))=\\
    &=&\lambda_G(\beta((xy)/u,u))^{-1}\lambda_G(\beta(x,y))\lambda_G(\beta(y/u,u)))=\\
    &=&j(\beta)_u(x,y)
\end{eqnarray*}
for every $x$, $y \in X$.
\end{proof}

We can now prove Theorem \ref{Th:Main2}. By Proposition \ref{Pr:SC}, a connected quandle $X$ is simply connected if and only if $H^2_c(X,\sym{S})=1$ for every set $S$. Let $X$ be a latin quandle. If $H^2_c(X,G)=1$ for every group $G$, then certainly $H^2_c(X,\sym{S})=1$ for every set $S$. Conversely, suppose that $H^2_c(X,\sym{S})=1$ for every set $S$, let $G$ be a group and let $\beta\in Z^2_c(X,G)$. Let $u\in X$. Since $H^2_c(X,\sym{G})=1$ and $j(\beta)\in Z^2_c(X,\sym{G})$, Corollary \ref{trivial cohomology and normalized} implies $j(\beta)_u=1$. By Proposition \ref{j commutes with normalization}, $\lambda_G(\beta_u(x,y)) = j(\beta_u)(x,y) = j(\beta)_u(x,y) = 1$ and therefore $\beta_u(x,y)=1$ for every $x$, $y\in X$. By Corollary \ref{trivial cohomology and normalized}, $H^2_c(X,G)=1$.

\begin{remark}
Let $X$ be a connected quandle. Are the following conditions equivalent?
\begin{itemize}
\item[(i)] $X$ is simply connected,
\item[(ii)] $H^2_c(X,G)=1$ for every group $G$.
\end{itemize}
\end{remark}

\subsection{Conjugacy quandle cocycle invariants}\label{Sc:CQCinvariants}

We conclude the paper by establishing Corollary \ref{Cr:Main}.

In \cite[Section 5]{CEGS}, a new family of knot invariants was defined by using constant quandle cocycles. Let $X$ be a quandle, $G$ a group and $\beta\in Z^2_c(X,G)$. Let  $(\tau_1,\ldots,\tau_k)$ be all the crossings of an oriented knot $K$ encountered while traveling around $K$ starting from some base point and following the orientation of $K$. For a crossing $\tau$ and coloring $\mathcal C\in\mathrm{col}_X(K)$, let $B(\tau,\mathcal C) = \beta(x_\tau,y_\tau)^{\epsilon_{\tau}}$, where $x_\tau$ is the color on the understrand, $y_\tau$ is the color on the overstrand, and $\epsilon_\tau$ is the sign of the crossing. Let $\varphi(K,\mathcal C) = \prod_{i=1}^k B(\tau_i,\mathcal C)$. For $g\in G$, let $[g]$ be the conjugacy class of $g$ in $G$. Then
\begin{displaymath}
    \varphi_{X,G,\beta}(K) = \setof{[\varphi(K,\mathcal C)]}{\mathcal C\in\mathrm{col}_X(K)}
\end{displaymath}
is a \emph{conjugacy quandle cocycle invariant} of $K$. According to \cite[Theorem 5.5]{CEGS}, this is indeed an invariant of oriented knots.

Let us prove Corollary \ref{Cr:Main}. Let $X$ be a simply connected latin quandle, $G$ a group and $\beta\in Z^2_c(X,G)$. By \cite[Proposition 5.6]{CEGS}, if $\beta$ is cohomologous to the trivial constant cocycle, then $\varphi_{X,G,\beta}(K)$ is trivial. It therefore suffices to show that $H^2_c(X,G)=1$ and this follows from Theorem \ref{Th:Main2}.


\section*{Acknowledgments}

We thank an anonymous referee for several important comments and in particular for pointing out the highly relevant preprint of Clauwens \cite{Clauwens}. The first author would like to thank Prof. David Stanovsk\'{y} for fruitful discussions during his visit to Ferrara in January 2016 from which this work was started and Dr. Giuliano Bianco who introduced him to the topic.

Marco Bonatto partially supported by Simons Foundation Collaboration Grant 210176 to Petr Vojt\v{e}chovsk\'y. Petr Vojt\v{e}chovsk\'y partially supported by University of Denver PROF grant.

\end{document}